\documentclass{article}
\usepackage{amssymb,amsmath,mathrsfs,latexsym,amscd}
\usepackage{exscale,latexsym,amsthm,graphics}

\usepackage{makeidx}
\makeindex
\usepackage{epsfig}
\usepackage{hyperref}
\usepackage{txfonts}
\newtheorem{thm}{Theorem}[section]
\newtheorem{cor}[thm]{Corollary}

\newtheorem{lemma}[thm]{Lemma}
\newtheorem{prop}[thm]{Proposition}
\newtheorem{proposition}[thm]{Proposition}

\newtheorem{defn}[thm]{Definition}
\newtheorem{definition}[thm]{Definition}

\newtheorem{remark}[thm]{Remark}
\newtheorem{example}[thm]{Example}

\def\bM {{\mathbb M}}  
 \def\Q {{\mathbb Q}}

\def\N {{\mathbb N}}

\def\qed{{\hfill $\Box$ \bigskip}}
\def\N {{\mathbb N}}
\def\R {{\mathbb R}}

\def\EE{{\mathbb E}}
\def\P{{\mathbb P}}
\newcommand{\F}{\mathcal{F}}

\def\E{{\mathcal E}}
\def\P{{\mathbb P}}

\numberwithin{equation}{section}
\begin{document}
\setcounter{secnumdepth}{5}
\setcounter{tocdepth}{5}
\noindent
{{\Large\bf On singular stochastic differential equations and Dirichlet forms}\footnote{This research was supported by NRF-DFG Collaborative Research program and Basic Science Research Program 
through the National Research Foundation of Korea (NRF-2012K2A5A6047864 and NRF-2012R1A1A2006987).}}

\bigskip
\noindent
{\bf Jiyong Shin},
{\bf Gerald Trutnau}
\\

\noindent
{\small{\bf Abstract.} This survey paper is a structured concise summary of four of our recent papers on the stochastic regularity of diffusions that are associated to regular strongly local (but not necessarily symmetric) Dirichlet forms. Here by stochastic regularity we mean the question whether a diffusion associated to a Dirichlet form as mentioned above can be started and identified as a solution to an explicit stochastic differential equation for explicitly given starting points. Beyond the stochastic regularity, we consider its applications to strong existence and pathwise uniqueness of singular stochastic differential equations.\\ 



\tableofcontents

\section{Introduction}
This survey paper is a summary of the main results of \cite{ShTr13b, ShTr13a, RoShTr, ShTr15}, which we present  systematically in concise structured form. Throughout, we consider a (non-)symmetric, strongly local, regular Dirichlet form  $\mathcal E$ on $L^{2}(E, m)$ where $E$ is a locally compact separable metric space and $m$ is a positive Radon measure on $(E, \mathcal{B}(E))$ with full support. We further assume that the symmetric part of $\mathcal E$ admits a carr\'e du champ. Our main concerns are the construction of a Hunt process associated to $\mathcal E$ that starts from as much as possible explicitly specified points in $E$ and subsequently the identification of the corresponding stochastic differential equation (hereafter SDE) for any of these starting points. Step by step we present methods to arrive at the identification of the corresponding SDE.\\
The first step is to find a pointwise heat kernel, i.e. the existence of a heat kernel $p_t(x,y)$ for all $x,y \in E$, associated with $\mathcal E$, and in the sequel to construct a Hunt process with the help of  the transition function 
of $p_t(\cdot,\cdot)$.  By association of $p_t(\cdot,\cdot)$ with $\mathcal E$, we mean that the $L^{2}(E, m)$-semigroup 
 of $\mathcal E$ coincides $m$-a.e. with the transition function of $p_t(\cdot,\cdot)$, i.e. the transition function of $p_t(\cdot,\cdot)$ induces an $L^{2}(E, m)$-semigroup that coincides with the one of $\mathcal E$. In accordance with the symmetric case we call this association Fukushima's absolute continuity condition. We explain two ways to obtain a pointwise heat kernel. In the symmetric case we adopt the method of \cite{St3} to obtain its existence. If the four conditions of Definition \ref{p;dcpo} are satisfied, then $p_t(x,y)$ exists, is locally H\"older in $(0,\infty)\times E \times E$ and satisfies the heat kernel estimate of Theorem \ref{t;2.7she}.  Moreover, the transition function is strong Feller (cf. Proposition \ref{p;strongfl}(i)). In the general, possibly non-symmetric case,  we consider the non-symmetric Dirichlet form given by the closure of  the bilinear form in (\ref{df}) below on $L^2(\R^d,m)$, $m:=\rho\,dx$, where the conditions on $A$, $\rho$, $B$ are formulated in {\bf(A1)}-{\bf(A3)} of Subsection \ref{ss;ermj}. Here, as a toy example we only consider the case where $A=id$, the case where $A$ is not the identity matrix can be treated similarly. 
We may then apply known elliptic regularity results from \cite[Theorem 5.1]{BGS} and  \cite[Theorem 1.7.4]{BKRS} (see Propositions \ref{p;gdegep2.3} and \ref{t;morrey2.4} below) and follow the main lines of \cite{AKR} to find a pointwise heat kernel. \\
The next step is to construct a Hunt process with given pointwise heat kernel $p_t(\cdot,\cdot)$. This construction is in general different from the construction of a Hunt process via the canonical scheme from a regular Dirichlet form which has only unique distributions for quasi-every starting point.  A well-known method to obtain a pointwise Hunt process is to show that the transition function of $p_t(\cdot,\cdot)$ induces a Feller semigroup. Here the conditions of Lemma \ref{t;Feller} appear to be the right ones 
in our framework since one can use the continuity of the heat kernel and estimates for it to verify these (see for instance Remark \ref{Fellerapplication} below). Another method to obtain a Hunt process with given pointwise heat kernel $p_t(\cdot,\cdot)$ is the Dirichlet form method. It is a refinement of the construction scheme introduced in \cite[Section 4]{AKR} to the case of  symmetric Dirichlet forms on a locally compact separable metric space that admit a carr\'e du champ. The method applies to certain non-symmetric cases as well, for instance to our toy example.\\
Once having constructed a Hunt process from the pointwise heat kernel, we aim at identifying it as a pointwise weak solution to a SDE. We explain two methods for its
identification. The first one is the well-known strict Fukushima decomposition (see Proposition \ref{p;LSFD} here, or \cite{fuku93} from where it originates, or in the monograph \cite[Theorem 5.5.5]{FOT}) and it applies in the symmetric case. Proposition \ref{p;LSFD} requires estimates on potentials coming from supersmooth measures that appear 
in the integration by parts formulas and in the energy for the Dirichlet form applied to the coordinate projections. Here Proposition \ref{p;smooth} in combination with Lemma \ref{l;miz} appear as very useful and make it possible to apply Fukushima's strict decomposition to a wide range of situations as we demonstrate by concrete examples in Subsections \ref{sdbm1} and  \ref{s;2admine}. However, in some cases the global estimate on the resolvent density which is obtained by taking the Laplace transform of the global estimate on the transition kernel density from Theorem \ref{t;2.7she} may not lead to satisfactory results as explained right after Lemma \ref{l;miz}. 
For these cases, we use a localization procedure that stems from \cite[Section 5]{ShTr13a}), but we formulate it here in more details and in a more general frame. It applies on open or closed subsets $E$ of the $d$-dimensional Euclidean space and involves part Dirichlet forms, Nash type inequalities (hence better local Gaussian heat kernel estimates) on a nice exhaustion up to a capacity zero set of $E$ by an increasing sequence of relatively compact open sets. In this localization procedure, described right after the paragraph that follows Lemma \ref{l;miz}, the strict Fukushima decomposition is applied only locally. The second method for the identification of the SDE is classical direct stochastic calculus. It is used in Subsection \ref{s;scftioft}. The drift corresponding to the coordinate projections is determined locally through the generator applied to smooth functions with compact support and the quadratic variation of the corresponding local martingale part 
can for instance be determined as in Proposition \ref{thm3.2}. For details we refer to the mentioned subsection.\\
Section \ref{s;strongun} is devoted to applications of stochastic regularity to strong existence and pathwise uniqueness of SDEs. We show that the weak solutions constructed in Subsections \ref{s;scftioft} and \ref{s;2admine} coincide with 
the strong and pathwise unique solutions that were constructed by probabilistic means up to their explosion times in \cite[Theorem 2.1]{KR} and \cite[Theorem 1.1]{Zh}. Thus if $\mathcal E$ is conservative and symmetric or if the corresponding transition function is strong Feller in the non-symmetric case, then the weak solutions obtained by stochastic regularity are non-explosive for any starting point (cf. \cite[Theorem 4.5.4(iv)]{FOT} and Remark \ref{r;connoex}). In particular, analytic conservativeness criteria that cover the whole framework of this paper can be found in \cite{GTr2016}. In this way, we provide new analytic non-explosion criteria for the strong and pathwise unique solutions of  \cite[Theorem 2.1]{KR} and \cite[Theorem 1.1]{Zh} which differ from the probabilistic non-explosion criteria presented in these papers.

\section{Preliminaries and construction of a Hunt process satisfying Fukushima's absolute continuity condition}\label{s2;sfdoh}
\subsection{Notations}
As a first general remark on our notations, we shall follow the monographs \cite{FOT}, \cite{o}, \cite{O13} and \cite{MR}. Thus should there be a notation that is not defined here, it can be found in these references.\\
For a locally compact separable metric space $(E,d)$ with Borel  $\sigma$-algebra $\mathcal{B}(E)$ we denote the set of all $\mathcal{B}(E)$-measurable $f : E \rightarrow \R$ which are bounded, or nonnegative by $\mathcal{B}_b(E)$, $\mathcal{B}^{+}(E)$ respectively.  The usual $L^q$-spaces  $L^q(E, \mu)$, $q \in[1,\infty]$ are equipped with $L^{q}$-norm $\| \cdot \|_{L^q(E, \mu)}$ with respect to the  measure $\mu$ on $E$  and $L^{q}_{loc}(E,\mu) := \{ f \,|\; f \cdot 1_{U} \in L^q(E, \mu),\,\forall U \subset E, U \text{ relatively compact open} \}$, where $1_A$ denotes the indicator function of a set $A \subset E$. 
If $\mathcal{A}$ is a set of functions $f : E \to \R$, we define $\mathcal{A}_0 : = \{f \in \mathcal{A} \ | $ supp($f$) : = supp($|f| dm$) is compact in $E \}$ and $\mathcal{A}_b$ : = $\mathcal{A} \cap \mathcal{B}_b(E)$. The inner product on $L^2(E, \mu)$ is denoted by $(\cdot,\cdot)_{L^2(E, \mu)}$. As usual, we also denote the set of continuous functions on $E$, the set of continuous bounded functions on $E$, the set of compactly supported continuous functions in $E$ by $C(E)$, $C_b(E)$, $C_0(E)$, respectively. The space of continuous functions on $E$ which vanish at infinity is denoted by $C_{\infty}(E)$. For $A \subset E$ let $\overline{A}$ denote the closure of $A$ in $E$, $A^c:=E\setminus A$.  We write  $B_{r}(y): = \{ x \in E \ | \ d(x,y) < r  \}$, $r>0$, $y \in E$.\\
Let $\nabla f : = ( \partial_{1} f, \dots , \partial_{d} f )$  and  $\Delta f : = \sum_{j=1}^{d} \partial_{jj} f$ where $\partial_j f$ is the $j$-th weak partial derivative of $f$ on $\R^d$ and $\partial_{ij} f := \partial_{i}(\partial_{j} f) $, $i,j=1, \dots, d$.  As usual $dx$ denotes Lebesgue measure on $\R^d$ and $\delta_x$ is the Dirac measure at $x$. Let $U \subset \R^d$, $d \ge 2$ be an open set. The Sobolev space $H^{1,q}(U, dx)$, $q \ge 1$ is defined to be the set of all functions $f \in L^{q}(U, dx)$ such that $\partial_{j} f \in L^{q}(U, dx)$, $j=1, \dots, d$, and 
$H^{1,q}_{loc}(U, dx) : =  \{ f  \,|\;  f \cdot \varphi \in H^{1,q}(U, dx),\,\forall \varphi \in  C_0^{\infty}(U)\}$. 
Here $C_0^{\infty}(U)$ denotes the set of all infinitely differentiable functions with compact support in $U$. We denote the set of all locally H\"{o}lder continuous functions of order $1-\alpha$ on $U$ on $U$ by $C^{1-\alpha}_{loc}(U)$, $0<\alpha<1$. For any $F \subset \R^d$, $F$ closed, let $C_0^{\infty}(F) : = \{f:F \to \R \ | \ \exists g \in C_0^{\infty}(\R^d), g|_F = f   \}$. If $F$ is compact, we also write $C^{\infty}(F)$ instead of $C_0^{\infty}(F)$. We equip $\R^d$ with the Euclidean norm $\| \cdot \|$ and the corresponding inner product $\langle \cdot, \cdot \rangle$. Let  $f^j (x):= x_j$, $j=1,\dots, d$, $x=(x_1,...,x_d) \in \R^d$, be the coordinate projections.

\subsection{The conditions {\bf (H1)} and {\bf (H2)}}\label{su;h1h2}
Throughout this paper, we consider a possibly non-symmetric, strongly local, regular Dirichlet form $(\E,D(\E))$   on $L^{2}(E, m)$ where $E$ is a locally compact separable metric space and $m$ is a positive Radon measure on $(E, \mathcal{B}(E))$ with full support on $E$ (see \cite{FOT}, \cite{o}, \cite{O13} and \cite{MR}).  As usual we define $\E_1(f,g) := \E(f,g) + (f,g)_{L^{2}(E , \, m)}$ for $f,g \in D(\E)$ and  $\| \,f\, \|_{D(\mathcal{E})} : = \E_1(f,f)^{1/2},  \;   f \in D(\E)$. Let  $(T_t)_{t > 0}$ (resp. $(\hat{T}_t)_{t > 0}$) and $(G_{\alpha})_{\alpha > 0}$ (resp. $(\hat{G}_{\alpha})_{\alpha > 0}$ ) be the strongly continuous contraction $L^2(E, m)$-semigroup (resp. cosemigroup) and resolvent (resp. coresolvent) associated to $(\E,D(\E))$ and $(L,D(L))$ (resp. $(\hat{L},D(\hat{L}))$) be the corresponding generator (resp. cogenerator) (see \cite[Diagram 3, p. 39]{MR}).  Then $(T_t)_{t>0}$ (resp. $(\hat{T}_t)_{t>0}$) and $(G_{\alpha})_{\alpha > 0}$ (resp. $(\hat{G}_{\alpha})_{\alpha > 0}$) are sub-Markovian (cf. \cite[I. Section 4]{MR}). Here an operator $S$ is called sub-Markovian if $0 \le f \le 1$ implies $0 \le Sf \le 1$. Then $(T_t)_{t >0}$ (resp. $(G_{\alpha})_{\alpha > 0}$) restricted to $L^1(E,m) \cap L^{\infty}(E,m)$ can be extended to strongly continuous contraction semigroups (resp. contraction resolvents) on all $L^r(E,m)$, $r \in [1,\infty)$. We denote the corresponding operator families again by $(T_t)_{t > 0}$ and $(G_{\alpha})_{\alpha > 0}$ and let $(L_r, D(L_r))$ be the corresponding generator on $L^r(E,m)$. Furthermore by \cite[I. Corollary 2.21]{MR}, it holds that $(T_t)_{t>0}$ is analytic on $L^2(E,m)$ and then by Stein interpolation (cf. e.g. \cite[Lecture 10, Theorem 10.8]{AV}) $(T_t)_{t>0}$ is also an analytic semigroup on $L^r(E,m)$ for all $r \in (2, \infty)$. Moreover, $(T_t)_{t>0}$ can be defined as a semigroup of contractions on $L^{\infty}(E,m)$, which is in general not strongly continuous. We denote the corresponding semigroup again by $(T_t)_{t>0}$. \\ \\
We consider the condition
\begin{itemize}
\item[{\bf (H1)}] There exists a $\mathcal{B}\left ( (0,\infty)\right )\otimes\mathcal{B}(E) \otimes \mathcal{B}(E)$ measurable non-negative map $p_{t}(x,y)$ such that
\begin{equation}\label{P1density}
P_t f(x) := \int_{E} p_t(x,y)\, f(y) \, m(dy) \,, \; t>0, \ \ x \in E,  \ \ f \in \mathcal{B}_b(E),
\end{equation}
is a (temporally homogeneous) sub-Markovian transition function (see \cite[Section 1.2]{CW}) and an $m$-version of $T_t f$ if $f  \in  L^2(E , m)_b$.
\end{itemize}
Here $p_{t}(x,y)$ is called the transition kernel density or heat kernel. Taking the Laplace transform of $p_{\cdot}(x, y)$, we see that $\bf{(H1)}$ implies that there exists a $\mathcal{B}(E) \otimes \mathcal{B}(E)$ measurable non-negative map $r_{\alpha}(x,y)$ such that
\[
R_{\alpha} f(x) := \int_{E} r_{\alpha}(x,y)\, f(y) \, m(dy) \,, \; \alpha>0, \; x \in E, f \in \mathcal{B}_b(E),
\]
is an $m$-version of $G_{\alpha} f$ if $f  \in  L^2(E, m)_b$. Here $r_{\alpha}(x,y)$ is called the resolvent kernel density.
For a signed Radon measure $\mu$ on $E$, let us define
\[
R_{\alpha} \mu (x) = \int_{E} r_{\alpha}(x,y) \, \mu(dy) \, , \;\; \alpha>0, \;\;x \in E,
\]
whenever this makes sense. Throughout this paper, we set $P_0 : = id$.\\ \\
Furthermore, assuming that $\bf{(H1)}$ holds, we can consider the condition
\begin{itemize}
\item[{\bf (H2)}] There exists a Hunt process with transition function $(P_t)_{t \ge 0}$.
\end{itemize}
We recall that $\bf{(H2)}$ means that there exists a Hunt process
\[
\bM = (\Omega , \mathcal{F}, (\mathcal{F}_t)_{t\geq0}, \zeta ,(X_t)_{t\geq0} , (\P_x)_{x \in E_{\Delta} }),
\]
with state space $E$ and the lifetime $\zeta:=\inf\{ t \ge 0 \ | \ X_t=\Delta\}$ such that $P_t(x,B) : = P_t 1_B (x) = \P_x(X_t \in B)$ for any $x \in E$, $B \in \mathcal{B}(E)$, $t \ge 0$ (cf. \cite{FOT}). Here, $\Delta$ is the cemetery point and as usual any function $f : E \rightarrow \R$ is extended to $\{\Delta\}$ by setting $f(\Delta):=0$. $E_{\Delta}: = E \cup \{\Delta\}$ is the one-point compactification if $E$ is not already compact, if $E$ is compact then $\Delta$ is added to $E$ as an isolated point. \\
By \cite[V. 2.12 (ii)]{MR}, it follows that $(\E, D(\E))$ is strictly quasi-regular. Then, by \cite[V.2.13]{MR} there exists a Hunt process 
\begin{equation}\label{HuntDF}
\tilde{\bM} = (\tilde{\Omega}, \tilde{\F}, (\tilde{\F})_{t \ge 0}, \tilde{\zeta}, (\tilde{X}_t)_{t \ge 0}, (\tilde{\P}_x)_{x \in E \cup \{ \Delta \} })
\end{equation}
(strictly properly) associated with $(\E, D(\E))$. It is here important to note that the transition function of $\tilde{\bM}$ will in general satisfy (\ref{P1density}) only for $m$-a.e. $x\in E$ (or quasi-every $x\in E$), even if $(T_t)_{t \ge 0}$ is strong Feller, i.e. $T_t f$ has a continuous $m$-version for any $f \in \mathcal{B}_b(E)$, because the Hunt process $\tilde{\bM}$ 
is unique only for quasi-every (hence in particular $m$-a.e) starting point (see for instance \cite[Theorems 4.2.8 and A.2.8.]{FOT}). Therefore a Hunt process as in $(\bf{H2})$ has to be explicitly constructed from the transition function in $\bf{(H1)}$. This will be done in Subsection \ref{sec2.4} below.
\begin{defn}\label{r;absolute}
If $\bf{(H1)}$ and $\bf{(H2)}$ hold, then we say that $\bM$ satisfies the {\it absolute continuity condition} (cf. \cite[(4.2.9)]{FOT} and also \cite[Theorem 3.5.4 (ii)]{O13}).
\end{defn}

\begin{remark}\label{r;connoex}
Let $\bM$ satisfy the absolute continuity condition. Suppose $(\E,D(\E))$ is conservative and $(P_t)_{t \ge 0}$ is strong Feller, i.e. for $t>0$ we have $P_t(\mathcal{B}_b(E)) \subset C_b(E)$. Then, since $m$ has full support, one can easily see that 
\[
\P_x(\zeta =\infty) =1, \quad \forall x \in E .
\]
\end{remark}

\subsection{The existence of a transition kernel density}\label{etradkd}

In this subsection, we illustrate two methods to find a transition kernel density as in $\bf{(H1)}$. The first method is from \cite{St3}. The second method depends on elliptic regularity results. In the symmetric case in \cite{BGS} it has been shown in a nice way how to obtain $\bf{(H1)}$ (and more) starting from an embedding of $D(L_p)$ for some $p>1$ into the space of continuous functions on compact subsets of $E$. The latter is naturally implied by elliptic regularity results via Sobolev embedding. Instead of formalizing the results of \cite{BGS} to the non-symmetric case in Subsection \ref{ss;ermj} right  after Proposition \ref{t;morrey2.4} below,  we follow a toy example that we continue in course of the subsequent sections. 

\subsubsection{Symmetric Dirichlet forms represented by a carr\'e du champ}\label{s2;stmg}

Throughout this subsection, we assume that $(\E,D(\E))$  is symmetric. Then 
$(\E,D(\E))$ can be written as 
\[
\E(f,g) = \frac{1}{2} \int_{E} \,d \mu_{\langle f,g\rangle}, \quad f,g \in D(\E),
\]
where $\mu_{\langle\cdot,\cdot \rangle}$ is a positive symmetric bilinear form on $D(\E) \times D(\E)$ with values in the signed Radon measures on $E$, called energy measures.
The positive measure $\mu_{\langle f,f\rangle}$ can be defined via the formula
\[
\int_{E} \phi \, d\mu_{\langle f,f\rangle} = 2\E (f,\phi f) - \E(f^2,\phi),
\]
for every $f \in D(\E)_b$ and every $\phi \in D(\E) \cap C_0(E)$. Let $D(\E)_{loc}$ be the set of all measurable functions $f$ on $E$ for which on every relatively compact open set $G \subset E$ there exists a function $g \in D(\E)$ with $f=g$ $m$-a.e on $G$. By an approximation argument we can extend the quadratic form $f \mapsto \mu_{\langle f,f\rangle}$ to $D(\E)_{loc} = \big\{ f \in L^2_{loc}(E, m) \, | \,\mu_{\langle f,f\rangle} \ \text{is a Radon}\big.$ $\big. \text{measure} \big\}$.  By polarization we then obtain for $f,g \in D(\E)_{loc}$ a signed Radon measure
\[
 \mu_{\langle f,g\rangle}= \frac{1}{4} (  \mu_{\langle f+g,f+g\rangle}- \mu_{\langle f-g,f-g\rangle} ).
\]
For these properties of energy measures we refer to \cite{FOT}, \cite[Proposition 1.4.1]{Le}, and \cite{Mo} (cf. \cite[Appendix]{St3}). In this article, whenever $\E$ is symmetric, we will assume that it admits a carr\'e du champ
\[
\Gamma: D(\E) \times D(\E) \to L^1(E,m)
\]
as in \cite[Definition 4.1.2]{BH}. This means 
\[
\mu_{\langle f,g\rangle}=\Gamma(f,g)\,dm
\]
i.e.  $\mu_{\langle f,g\rangle}$ is absolutely continuous with respect to $m$ with density 
$\Gamma(f,g)$ for any $f,g\in D(\E)$. The energy measures $\mu_{\langle f,f\rangle}$ or equivalently the carr\'e du champ operator, define in an intrinsic way a pseudo metric $\gamma$ on $E$ by
\[
\gamma(x,y) = \sup \Big\{f(x) - f(y) \ | \ f  \in D(\E)_{loc} \cap C(\R^d), \  \Gamma(f,f)   \le   1 \, m\text{.a.e. on } E  \Big\}, 
\]
(cf. \cite{BM}).  We define the balls with respect to the intrinsic metric by 
\[
\tilde{B}_r(x) = \{ y \in E \ | \ \gamma(x,y) < r \}, \quad x \in E, \quad r>0.
\]
\begin{defn}\label{p;dcpo}
\begin{itemize}
\item[(i)]$(\E,D(\E))$ is called strongly regular if $\gamma(\cdot,\cdot)$ is a metric on $E$ whose topology coincides with the original one.
\item[(ii)]
We say the completeness property holds, if for all balls $\tilde{B}_{2r}(x)\subset E$,  $x \in E$, $r>0$, the closed balls $\overline{\tilde{B}_r(x)}$ are complete (or equivalently, compact).
\item[(iii)]
We say the doubling property holds if there exists a constant N= N(E) such that for all balls $\tilde{B}_{2r}(x) \subset E$
\[
m (\tilde{B}_{2r}(x)) \le 2^N  m (\tilde{B}_{r}(x)).
\] 
\item[(iv)] We say the (scaled) weak Poincar\'e inequality holds, if there exists a constant $C_p= C_p(E)$ such that for all balls $\tilde{B}_{2r}(x) \subset E$
\[
\int_{\tilde{B}_r(x)} |f - \tilde{f}_{x,r} |^2 \, dm \le C_p \ r^2 \int_{\tilde{B}_{2r}(x)} \ \Gamma(f,f)\,dm, \quad \forall f \in D(\E),
\]
where  $\tilde{f}_{x,r} = \frac{1}{m (\tilde{B}_r(x))}\int_{\tilde{B}_r(x)} f \, dm$.
\end{itemize}
\end{defn}
Suppose $(\E,D(\E))$ satisfies the properties (i)-(iv) of Definition \ref{p;dcpo}. Then by \cite[p. 286 A)]{St3} with $Y=E$, there exists a jointly continuous transition kernel density  $p_{t}(x,y)$, locally H\"older continuous in $(t,x,y)\in (0,\infty)\times E \times E$ (see \cite[Proposition 3.1)]{St3}), such that
\begin{equation}\label{e;tdst3}
P_t f(x) := \int_{E} p_t(x,y)f(y) \, m(dy), \ \  t>0, \ x,y \in E, \ f\in \mathcal{B}_b(E)
\end{equation}
is an $m$-version of $T_t f$ if $f  \in  L^2(E, m)_b$. In particular, condition $\bf{(H1)}$ is satisfied. Furthermore, we obtain from \cite[Corollary 4.2)]{St3} the following estimate of transition kernel density:
\begin{thm}\label{t;2.7she}
Suppose $(\E,D(\E))$ satisfies the properties (i)-(iv) of Definition \ref{p;dcpo}. Then, given any $\varepsilon>0$, for all points $x,y \in E$ and all $t > 0$
\begin{equation}\label{e;hkeos}
p_t(x,y) \le C\frac{1}{\sqrt{m(\tilde{B}_{\sqrt{t}}(x))}} \cdot  \frac{1}{\sqrt{m(\tilde{B}_{\sqrt{t}}(y))}}  \cdot \exp \left( - \frac{\gamma^2(x,y)}{(4+\varepsilon)t} \right),
\end{equation}
where $C$ is a constant depending only on $N=N(E)$ and $C_p=C_p(E)$.
\end{thm}
Using Theorem \ref{t;2.7she}, exactly as in \cite[Proposition 3.3]{ShTr13a}, we can show: 
\begin{proposition}\label{p;strongfl}
Suppose $(\E,D(\E))$ satisfies the properties (i)-(iv) of Definition \ref{p;dcpo}. Then:
\begin{itemize} 
\item[(i)] $(P_t)_{t \ge 0}$ and $(R_{\alpha})_{\alpha > 0}$) are strong Feller.
\item[(ii)] $(\textbf{H1})$ and \textbf{(H2)}$^{\prime}$(iii), (iv) from Subsection \ref{2.4.2.1} below hold for $(P_t)_{t \ge 0}$.
\item[(iii)] Suppose $E= \R^d$ and $C^{-1}\|x-y\|\le \gamma(x,y)\le C \|x-y\|$ for some constant $C\ge 1$ and any $x,y\in \R^d$. Then $P_t (L^1(\R^d,m)_0) \subset C_{\infty} (\R^d)$.
\end{itemize}
\end{proposition}
According to \cite[Theorem 4]{St1}, \cite[Theorems 3.1(i),(ii), and 3.6]{St5}, we have the following conservativeness criterion: 
\begin{thm}\label{t;cstconsv}
 Suppose $(\E,D(\E))$ satisfies properties (i) and (ii) of Definition \ref{p;dcpo} and that $\gamma(x,y)<\infty$ for all $x,y\in E$. Then $(\E,D(\E))$ is conservative, if
\[
\int_{1}^{\infty} \frac{r}{\log \big (m (\tilde{B}_{r}(x_0)  \big )} \, dr =\infty,
\]
where $x_0 \in E$ is arbitrary but fixed.
\end{thm}

\subsubsection{Using elliptic regularity}\label{ss;ermj}
We consider the following conditions {\bf (A1)}-{\bf (A3)} in dimension $d\ge 2$:
\begin{itemize}
\item[{\bf (A1)}]
$\rho = \xi^2$, $\xi \in H^{1,2}_{loc}(\R^d, dx)$,  $\rho > 0 \ \ dx$-a.e. and 
\[
\frac{\| \nabla \rho \|}{\rho} \in L^{p}_{loc} (\R^d, m), \quad m := \rho dx,
\]
$p:=d + \varepsilon$ for some $\varepsilon >0$, $a_{ij}=a_{ji}\in H^{1,p}_{loc}(\R^d,dx)$, $1\le i,j\le d$ and the matrix $A= (a_{ij})_{1 \le i,j \le d}$ is locally strictly elliptic $dx$-a.e. on $\R^d$, i.e. for each compact set $K \subset \R^d$, there exists some $\kappa_K > 0$ such that  $\kappa_K \| \xi \|^2 \le  \langle A(x) \xi, \xi \rangle$, $\forall \xi \in \R^d$,  $dx$-a.e. $x\in K$. 
\end{itemize}
By {\bf (A1)} the symmetric positive definite bilinear form
\[
\E^0(f,g) : = \frac{1}{2} \int_{\R^d} \langle A \nabla f ,  \nabla g \rangle \, dm , \quad f, g \in C_0^{\infty}(\R^d)
\]
is closable in $L^2(\R^d,m)$ and its closure $(\E^0,D(\E^0))$ is a symmetric, strongly local, regular Dirichlet form. We further assume
\begin{itemize}
\item[{\bf (A2)}] $B : \R^d \to \R^d, \ \|B\| \in L_{loc}^{p}(\R^d, m)$ where $p$ is the same as in {\bf (A1)} and
\[
\int_{\R^d} \langle B,\nabla f \rangle \, dm = 0, \quad \forall f \in C_0^{\infty}(\R^d),
\]
\end{itemize}
and
\begin{itemize}
\item[{\bf (A3)}]
\[
\left| \int_{\R^d} \langle B,\nabla f \rangle \ g \ \rho \ dx \right| \le c_0 \ \E^0_1 (f,f)^{1/2} \ \E^0_1 (g,g)^{1/2}, \quad \forall f,g \in C_0^{\infty}(\R^d),
\]
where $c_0$ is some constant (independent of $f$ and $g$).
\end{itemize}
Next, we consider the non-symmetric bilinear form 
\begin{equation}\label{df}
\E(f,g) : = \frac{1}{2} \int_{\R^d} \langle A \nabla f , \nabla g \rangle \, dm - \int_{\R^d} \langle B,\nabla f \rangle \ g \  dm, \quad
 f, g \in C_0^{\infty}(\R^d)
\end{equation}
in $L^2(\R^d,m)$. 
Then by {\bf (A1)}-{\bf (A3)} $(\E, C_0^{\infty}(\R^d))$ is closable in $L^2(\R^d,m)$ and the closure $(\E,D(\E))$ is a non-symmetric Dirichlet form (cf. \cite[II. 2.d)]{MR}), which is strongly local and regular. \\
We now state the elliptic regularity result  \cite[Theorem 5.1] {BGS}, which is based on results of \cite{BKR1}, \cite{BKR2}, but improves them. \cite[Theorem 5.1] {BGS} is formulated for general open subsets $U\subset \R^d$, but we shall only be concerned with $U=\R^d$.
\begin{prop}\label{p;gdegep2.3}
Let $d \ge 2$ and $\mu$ a locally finite (signed) Borel measure on $\R^d$ that is absolutely continuous with respect to the Lebesgue measure $dx$ on $\R^d$. Suppose $A= (a_{ij})_{1 \le i,j \le d}$ is as in {\bf (A1)}. Let either $h_i$, $c \in L_{loc}^p(\R^d,dx)$ or $h_i$, $c \in L^p_{loc}(\R^d, \mu)$ and let $f \in L^p_{loc}(\R^d,dx)$. Assume that 
\[
\int_{\R^d} \Big( \sum_{i,j = 1}^{d} a_{ij} \partial_{ij} \varphi + \sum_{i=1}^{d} h_i \partial_i \varphi + c \varphi  \Big) \  d\mu = \int_{\R^d}  \varphi  f \, dx, \quad \forall \varphi \in C_0^{\infty} (\R^d),
\]
where $h_i$, $c$ are locally $\mu$-integrable. Then $\mu$ has a density in $H^{1,p}_{loc}(\R^d)$ that is locally H\"older continuous.
\end{prop}
Additionally, we restate Morrey's estimate in our setting  (see \cite[Theorem 1.7.4 ]{BKRS}).
\begin{prop}\label{t;morrey2.4}
Assume $p > d \ge 2$. Let V be a bounded domain in $\R^d$ and $b : V \to \R^d$ and $c,e : V \to \R$ such that
\[
h_i \in L^p (V, dx) \ \  \text{and} \ \ c,e \in L^q(V,dx) \ \ \text{for} \ \ q:=\frac{dp}{d+p}>1.
\]
Let $a_{ij} = a_{ji}$, $a_{ij} \in H^{1,p}_{loc}(\R^d,dx) $ for all $1 \le i,j \le d$ and $\kappa^{-1} \ \| \xi \|^2 \le \langle A(x) \xi,\xi \rangle \le \kappa \| \xi \|^2$, $\forall \xi \in \R^d$, $x \in V$ for some $\kappa \ge 1$. Assume that $u \in H^{1,p}(V)$ is a solution of 
\[
\int_{V} \sum_{i=1}^d \Big( \partial_i \varphi \Big( \sum_{j=1}^d a_{ij} \partial_j u + h_i u  \Big)  \Big) + \varphi (cu + e) \, dx =0, \quad \forall \varphi \in C_0^{\infty}(V),
\]
Then for every domain $V'$ with $V' \subset \overline{V'} \subset V$, we obtain the estimate
\[
\| u \|_{H^{1,p}(V')} \le c (\| e \|_{L^q(V,dx)} + \| u \|_{L^1(V,dx)}),
\]
where  $c < \infty$ is some constant independent of $e$ and $u$.
\end{prop}
The elliptic regularity results of Propositions \ref{p;gdegep2.3} and \ref{t;morrey2.4} have been applied in the symmetric case, i.e. $B\equiv 0$  in \cite{BGS} and in particular  ${\bf (H1)}$,  ${\bf (H2)}$, up to the solution of a corresponding martingale problem have been derived in this situation. We refer the interested reader to the mentioned article. Propositions \ref{p;gdegep2.3} and \ref{t;morrey2.4} and the elliptic regularity results of  \cite{BKR1}, \cite{BKR2}, 
can also be applied in the non-symmetric case. This has been done in case $A=(a_{i,j})_{1\le i,j\le d}$ is the identity matrix in \cite{RoShTr} and we will consider this case as a toy example that we will continue throughout this article.  
From now on up to the end of this subsection, we shall hence assume that
\begin{equation}\label{aidass}
a_{ij}=\delta_{ij},\ 1\le i,j \le d,
\end{equation}
where $\delta_{ij}\in \{0,1\}$ is the Kronecker symbol, i.e. $A$ is the identity matrix.\\
Since by {\bf (A1)}, {\bf (A2)}, $\left\| \frac{\nabla \rho}{2 \rho} \right\|$, $\|B\| \in L^{p}_{loc}(\R^d,m)$, we get  $C_0^{\infty}(\R^d) \subset D(L_r)$ for any $r \in [1,p]$ and
\begin{equation}\label{operatorrepresentation}
L_r u = \frac{1}{2}\Delta u +  \langle \frac{\nabla \rho}{2 \rho}+B, \nabla u \rangle, \quad u \in C_0^{\infty}(\R^d), \quad r \in [1,p].
\end{equation}
In particular 
$$
\int_{\R^d} Lu \,dm=0, \quad u \in C_0^{\infty}(\R^d).
$$
Thus by Proposition \ref{p;gdegep2.3} and Sobolev embedding:
\begin{cor}\label{cor1.2} 
$\rho$ is in $H^{1,p}_{loc}(\R^d,dx)$ and has hence a continuous $dx$-version in $C_{loc}^{1-d/p}(\R^d)$.
\end{cor}

We shall always consider the continuous $dx$-version of $\rho$  and denote it also by $\rho$. Under the assumptions {\bf (A1)}-{\bf (A3)} we apply  Proposition \ref{p;gdegep2.3} and Proposition \ref{t;morrey2.4} to 
\[
\int (\alpha - \hat{L})u \ G_{\alpha} g \ \rho \ dx = \int u \ g \ \rho \ dx, \quad \forall u \in C_0^{\infty}(\R^d),
\]
where
\[
\hat{L} u = \frac{1}{2} \Delta u +  \langle \frac{\nabla \rho}{2 \rho} - B, \nabla u \rangle. 
\]
Doing this, we get (cf. \cite{RoShTr}):
\begin{cor}\label{cor1.3}
\begin{itemize} Let $\alpha > 0$, $t>0$, and $r \in [p,\infty)$. Then:
\item[(i)] For $g \in L^r(\R^d, m)$, we have
\[
\rho \ G_{\alpha} g \in H_{loc}^{1,p}(\R^d,dx)
\]
and for any open balls $B'\subset \overline{B'}\subset B \subset \overline{B}  \subset \{\rho > 0\}$ there exists $c_{B,\alpha} \in (0,\infty)$, independent of $g$, such that 
\begin{equation}\label{conresol}
\| \ \rho \ G_{\alpha} g \ \|_{H^{1,p}(B',dx)} \le c_{B,\alpha} \ \Big(\|G_{\alpha} g \|_{L^1(B,m)} + \| g \|_{L^{p}(B,m)} \Big).
\end{equation}
\item[(ii)] For $u \in D(L_r)$, we have 
\[
\rho \ T_t u \in H^{1,p}_{loc} (\R^d, dx)
\]
and for any open balls $B'\subset \overline{B'}\subset B \subset \overline{B}  \subset \{ \rho > 0\}$ there exists $c_B \in (0,\infty)$ (independent of $u$ and $t$) such that
\begin{eqnarray}\label{eq1.4}
\| \rho \ T_t u\|_{H^{1,p}(B',dx)} &\le& c_B \left( \| T_t u \|_{L^1(B,m)} + \| T_t (1-L_r) u  \|_{L^p(B,m)}   \notag    \right)\\
&\le& c_B \left( m(B)^{\frac{r-1}{r}} \| u \|_{L^r (\R^d,m) }  + m(B)^{\frac{r-p}{rp}} \| (1-L_r) u \|_{L^r(\R^d,m)}    \right).
\end{eqnarray}

\item[(iii)] Let $f \in L^r(\R^d,m)$. Then the above statements still hold with \eqref{eq1.4} replaced by
\[
\|\rho \ T_t f\|_{H^{1,p}(B',dx)} \le \tilde{c}_B \ (1+ t^{-1}) \| f \|_{L^r(\R^d,m)},
\]
where $\tilde{c}_B \in (0,\infty)$ (independent of $f$, $t$).
\end{itemize}
\end{cor}

\begin{remark}\label{rem1.5}
By \eqref{eq1.4} and Sobolev imbedding, for $r \in [p, \infty)$, $R>0$ the set
\[
\{T_t u  \ | \ t>0, \ u \in D(L_r), \ \|u\|_{L^r(\R^d,m)} + \|L_r u\|_{L^r(\R^d, m)} \le R \} 
\]
is equicontinuous on $\{\rho>0\}$.
\end{remark} 

Now by Corollaries \ref{cor1.2}, \ref{cor1.3}, and Remark \ref{rem1.5}, exactly as in \cite[section 3]{AKR} (cf. \cite{RoShTr}), we obtain:
\begin{thm}\label{t;2.6tdsr}
\begin{itemize}
\item[(i)] There exists a transition kernel density $p_t(\cdot,\cdot)$ on the open set 
$$
E: = \{\rho > 0 \}
$$
such that
\[
P_t f(x) : = \int_{\R^d} f(y) p_t(x,y) \, m(dy), \quad x \in E, \ t > 0
\]
is a (temporally homogeneous) sub-Markovian transition function and an $m$-version of $T_t f $ for any $f \in  \cup_{r \ge p} L^r(\R^d,m)$. 
\item[(ii)] $(P_t)_{t>0}$ is a semigroup of kernels on $E$ which is $L^r(\R^d,m)$-strong Feller for all $r \in [p,\infty)$, i.e.
\[
P_t f \in C(E), \quad \forall f \in \cup_{r \ge p} L^r(\R^d,m), \quad \forall t > 0.
\]
\item[(iii)]
\[
\lim_{t \to 0} P_{t+s} f(x) = P_s f(x), \quad \forall s \ge 0, \ x \in E, \ f \in C_0^{\infty}(\R^d).
\]
\item[(iv)]
$(P_t)_{t > 0}$ is a measurable semigroup on $E$, i.e. for $f \in \mathcal{B}^+ (\R^d)$ the map $(t,x) \mapsto P_t f(x)$ is $\mathcal{B}([0,\infty) \times E)$-measurable.

\end{itemize}
\begin{itemize}
\item[(v)] There exists a  resolvent kernel density $r_{\alpha}(\cdot, \cdot)$ defined on $E$ such that
 \[
R_{\alpha} f(x) : = \int f(y)\, r_{\alpha} (x,y) \, m(dy),  \quad x \in E, \ \alpha > 0,
\]
satisfies $\ R_{\alpha} f = G_{\alpha} f  \ \  m\text{-a.e for any} \  f \in \cup_{r \ge p}  L^r(\R^d,m)$ and $\alpha R_{\alpha}1 (x) \le 1$.
\item[(vi)] $(R_{\alpha})_{\alpha> 0}$ is a resolvent of kernels on $E$ and $(R_{\alpha})_{\alpha> 0}$ is  $L^r(\R^d,m)$-strong Feller for all $r \in [p,\infty)$, i.e.  $R_{\alpha} f \in C_b (E)$ for all $f \in \mathcal{B}_b(\R^d)$, and $R_{\alpha} f \in C (E)$ for all $f \in \cup_{r \ge p} L^r(\R^d,m)$.
\item[(vii)] Let $\alpha > 0$. Then for all $f \in \mathcal{B}_b(\R^d ) \cup \mathcal{B}^{+} (\R^d)$ and all $x \in E$
\[
R_{\alpha} f (x) = \int_0^{\infty} e^{- \alpha t} P_t f (x) \ dt.
\]
\item[(viii)] For all $u \in C_0^{\infty}(\R^d)$
\[
\lim_{\alpha \to \infty} \alpha R_{\alpha} u(x) = u(x) \quad \forall x \in E.
\]
\end{itemize}
\end{thm}
Note that applying  Corollary \ref{cor1.3}, we obtain in Theorem \ref{t;2.6tdsr} a locally H\"older continuous $m$-version $P_t f$  of $T_t f$  only on $E=\{\rho >0\}$, because the product $\rho T_t f$ has a locally H\"older continuous $m$-version and $\rho$ is H\"older continuous. The same holds for the locally  H\"older continuous $m$-version $R_{\alpha}f$ of $G_{\alpha} f$.\\
In order to show that condition ${\bf (H1)}$ holds (on $E=\{\rho>0\}$) we still need some preparations. Consider the strict capacity Cap$_{\E}$ of the non-symmetric Dirichlet form $(\E,D(\E))$ as defined in \cite[V.2.1]{MR} and \cite[Definition 1]{Tr5}, i.e. 
\[
\text{Cap}_{\E} =  \text{cap}_{1,\hat{G}_1\varphi}
\]
for some fixed $\varphi \in L^1(\R^d,m) \cap \mathcal{B}_b(\R^d)$, $0 < \varphi \le 1$. Let Cap be the capacity related to the symmetric Dirichlet form ($\E^0,D(\E^0)$) as defined in \cite[Section 2.1]{FOT}. It is known from  \cite[Theorem 2]{fuku85} that Cap$(\{\rho=0\})=0$. Then the following has been shown in \cite[Lemma 2.10]{RoShTr}:   
\begin{lemma}\label{lem2.8}
Let $N \subset \R^d$. Then 
\[
\emph{Cap}(N) = 0  \Rightarrow \emph{Cap}_{\E}(N) = 0.
\]
In particular $\emph{Cap}_{\E}(\{\rho = 0\})=0$.
\end{lemma}
The result of Lemma \ref{lem2.8} is intuitively clear by \cite[Lemma 2.2.7(ii)]{FOT} and condition {\bf (A2)}. In particular it implies that for the Hunt process (\ref{HuntDF}) it holds $\tilde{\P}_x(\tilde \sigma_{E^c}<\infty)=0$ for $m$-a.e. $x\in \R^d$ (actually for $\E$-q.-e. $x$, see \cite[IV. Proposition 5.30]{MR}), where $\tilde \sigma_{E^c}:=\inf\{t>0\,|\, \tilde{X}_t \in E^c\}$.\\
Let $(\E^{E},D(\E^{E}))$ denote the part Dirichlet form on $E$ of $(\E,D(\E))$ given by (\ref{df}) with $A=(a_{ij})_{1 \le i,j \le d}$ satisfying (\ref{aidass}). Let $(T_{t}^{E})_{t>0}$ denote its $L^2(E,m)$-semigroup.
By \cite[Theorem 3.5.7]{O13} the part process $(\tilde{X}_t^E)_{t \ge 0}$ of the Hunt process  (\ref{HuntDF}) is associated to $(\E^{E},D(\E^{E}))$. Hence for any $f \in \mathcal{B}_{b}(E)_0$ and $m$-a.e. $x \in E$
\begin{eqnarray}\label{partid}
&&T_{t}^{E} f(x) =  \tilde{\EE}_x [f(\tilde{X}^{E}_t), \; t < \tilde\sigma_{E^{c}} ] = \tilde{\EE}_x [f(\tilde{X}_t), \; t < \sigma_{E^{c}} ]
= \tilde{\EE}_x [f(\tilde{X}_t)] = T_{t}f(x)\nonumber \\
&&=  \int_E f(y) \, p_{t}(x, y)\,m(dy),
\end{eqnarray}
where the second equality follows from the definition of part process, the third since Cap$_{\E}(E^c)=0$ and the last since $f$ is in particular in $L^{p}(E, m)$. Extending 
$$
P_t f(x) : = \int_{E} f(y) p_t(x,y) \, m(dy), \quad x \in E, \ t > 0
$$
to $f\in L^1(E, p_t(x,\cdot) \,dm)\supset L^2(E,m)_b$, we see that condition $(\bf{H1})$ holds for the part Dirichlet form $(\E^{E},D(\E^{E}))$.

\begin{remark}
By using elliptic regularity results we do not necessarily obtain condition $\bf{(H1)}$ for the original Dirichlet form $(\E,D(\E))$ on $\R^d$. Instead, we have to exclude a capacity zero set from the state space. The new state space $E=\R^d\setminus \{\rho =0\}$ will then be an invariant set for the corresponding stochastic process that will be constructed and identified below in subsections \ref{sec2.4} (see Theorem \ref{existhunt}) and \ref{s;nsdbm}, i.e. we will be able to start and identify the corresponding SDE for every point in $E$ and the process will remain in $E$ until its lifetime. This is in contrast to the construction via the Feller method which can allow entrance boundaries, i.e. capacity zero sets from which the process can be started and identified but to which it will never return. As an example, we mention for instance  Proposition \ref{p;3.8}(i). Here we can start and identify the SDE (\ref{equast}) for every starting point $x\in \R^d$, $\alpha \in (-d+1,1)$, but by \cite[Example 3.3.2]{FOT} Cap$(\{0\})=0$, if and only if $\alpha\in [-d+2,d)$. Thus in the situation of Proposition \ref{p;3.8}(i), $0$ is an entrance boundary for any $\alpha\in [-d+2,1)$.

\end{remark}

\subsection{Construction of a Hunt process with given transition kernel density}\label{sec2.4}

In this subsection, we illustrate two methods to obtain $\bM$ as in Definition \ref{r;absolute} starting from assumption $\bf{(H1)}$. Concerning the second method in Subsection \ref{2.4.2.2}, we continue our toy example from Subsection \ref{ss;ermj} to explain the non-symmetric case. 

\subsubsection{The Feller method}\label{s;fema2.2}

Assuming $\bf{(H1)}$, a Hunt process as in $\bf{(H2)}$ can be constructed by means of a Feller semigroup. For the definition of Feller semigroup, we refer to \cite[Section 2.2]{CW}.
\begin{remark}\label{r;feller}
Under $\bf{(H1)}$, $(P_t)_{t \ge 0}$ is a Feller semigroup, if 
\begin{itemize}
\item[(i)] $\forall f \in C_{\infty}(E)$, $\lim_{t \to 0} P_t f= f$ uniformly on $E$,
\item[(ii)] $P_t C_{\infty}(E) \subset C_{\infty}(E)$ for each $t > 0$.
\end{itemize}
\end{remark}
It is well known that the condition of uniform convergence in Remark \ref{r;feller} (i) can be relaxed to pointwise convergence (see for instance \cite[Section 2.2 Exercise 4.]{CW}). 
The conditions of Remark \ref{r;feller} can be further relaxed to the conditions of the following lemma which are suitable for us. 
\begin{lemma}\label{t;Feller}
Suppose $\bf{(H1)}$ and that
\begin{itemize}
\item[(i)] $\lim_{t \to 0} P_t f(x) = f(x)$ for each $x \in E$ and $f \in C_0(E)$,
\item[(ii)] $P_t C_0(E) \subset C_{\infty}(E)$ for each $t>0$.
\end{itemize} 
Then $(P_t)_{t \ge 0}$ is a Feller semigroup. In particular $\bf{(H2)}$ holds  (cf. \cite[(9.4) Theorem]{BlGe}).
\end{lemma}

\begin{remark}\label{Fellerapplication}
One can use heat kernel estimates for $p_t(x,y)$ to check the assumptions of Lemma \ref{t;Feller}(i), (ii) (see \cite[proofs of Proposition 3.3(iii) and Lemma 3.6(i)]{ShTr13a} and the corresponding statement here in Lemma \ref{l;Feller}(i) in Section \ref{s;3.1sfdm} below).
\end{remark}

\subsubsection{The Dirichlet form method}\label{s;dm2.3}
The second method to obtain a Hunt process as in Definition \ref{r;absolute} starting from assumption $(\bf{H1})$ is by a method that we call the Dirichlet form method. 

\paragraph{The symmetric case}\label{2.4.2.1}
Throughout this subsection let $(\E,D(\E))$ be symmetric.  We assume further that $\E$ admits a carr\'e du champ
$\Gamma: D(\E) \times D(\E) \to L^1(E,m)$ as in Subsection \ref{s2;stmg}. \\
Consider the condition
\begin{itemize}
\item[{\bf (H2)$^{\prime}$}] We can find $\{ u_n \ | \ n \ge 1 \} \subset D(L) \cap C_0(E)$ satisfying:
\begin{itemize}
\item[(i)] For all $\varepsilon \in \Q \cap (0,1)$ and
$y \in D$, where $D$ is any given countable dense set in $E$, there exists $n \in \N$ such that $u_n (z) \ge 1$, for all $z \in \overline{B_{\frac{\varepsilon}{4}}(y)}$ and $u_n \equiv 0$ on $E \setminus B_{\frac{\varepsilon}{2}}(y)$.
\item[(ii)] $R_1\big( [(1 -L) u_n]^+ \big)$, $R_1\big( [(1 -L) u_n]^- \big)$, $R_1 \big( [(1-L_1)u_n^2]^+ \big)$, $R_1 \big( [(1-L_1)u_n^2]^- \big)$ are continuous on $E$ for all $n \ge 1$.
\item[(iii)] $R_1 C_0(E) \subset C(E)$.
\item[(iv)] For any $f \in C_0(E)$ and $x \in E$, the map $t \mapsto P_t f(x)$ is right-continuous on $(0,\infty)$.
\end{itemize}
\end{itemize}
Note that $L_1 u_n^2$, $n\ge 1$, in $({\bf H2})'(ii)$ is well-defined, since $D(L_1)_b$ is an algebra by \cite[I. Theorem 4.2.1]{BH} and $D(L)_{0,b}\subset D(L_1)_b$ by \cite[Lemma 2.5(ii)]{ShTr13a}. \\
The following is the main result of \cite[Section 2.1.2]{ShTr13a}. It is a refinement of the results obtained in a concrete symmetric situation in \cite[Section 4]{AKR}  to the case of strongly local regular symmetric Dirichlet forms that admit a carr\'e du champ:
\begin{proposition}\label{l;2.10}
Assume \textbf{(H1)} holds. Then $\textbf{(H2)}^{\prime}$ implies \textbf{(H2)}.
\end{proposition}
By an obvious application of Proposition \ref{l;2.10}, we obtain:
\begin{remark}
If $(T_t)_{t \ge 0}$ is strong Feller and $\textbf{(H2)}^{\prime}$(i)-(ii) and $\textbf{(H2)}^{\prime}$(iv)
hold for the corresponding transition function $(P_t)_{t \ge 0}$ and resolvent $(R_{\alpha})_{\alpha>0}$, 
then (\textbf{H1}) and (\textbf{H2}) hold.
\end{remark}

\paragraph{The non-symmetric case}\label{2.4.2.2}
This subsection is a continuation of Subsection \ref{ss;ermj} where we considered the Dirichlet form $(\E, D(\E))$ of (\ref{df}) with $A=(\delta_{ij})$.  We consider again the strict capacity Cap$_{\E}$ and the Hunt process (\ref{HuntDF}). Due to the properties of smooth measures with respect to Cap$_{\E}$ in \cite[Section 3]{Tr5} it is possible to consider the work \cite{Tr2} with cap$_{\varphi}$ (as defined in \cite{Tr2}) replaced by Cap$_{\E}$. In particular \cite[Theorem 3.10 and Proposition 4.2]{Tr2} apply with respect to the strict capacity Cap$_{\E}$ and therefore the paths of $\tilde{\bM}$ are continuous $\tilde{\P}_x$-a.s. for strictly $\E$-q.e. $x \in \R^d$ on the one-point-compactification $\R^d_{\Delta}$ of $\R^d$ with $\Delta$ as point at infinity.
We may hence assume that 
\begin{equation}\label{contipath}
\tilde{\Omega} = \{\omega = (\omega (t))_{t \ge 0} \in C([0,\infty),\R^d_{\Delta}) \ | \ \omega(t) = \Delta \quad \forall t \ge \tilde{\zeta}(\omega) \}
\end{equation}
and
\[
\tilde{X}_t(\omega) = \omega(t), \quad t \ge 0.
\]
Since $\text{Cap}_{\E}(\{\rho=0\})=0$ by Lemma \ref{lem2.8}, we obatain (cf. \cite{RoShTr}):
\begin{lemma}\label{nestconti}
Let $(F_k)_{k \ge 1}$ be an increasing sequence of compact subsets of $E: = \{\rho > 0\} $ with $\cup_{k \ge 1} F_k = E$ and such that $F_k \subset   \mathring{F}_{k+1}$, $k \ge 1$(here $\mathring{F}$ denotes the interior of $F$). Then
\[
\tilde{\P}_x (\tilde{\Omega}_0) = 1 \ \text{for strictly} \ \E \text{-q.e.} \ x \in E,
\]
where 
\[
\tilde{\Omega}_0 : = \tilde{\Omega} \cap \{\omega \ | \ \omega(0) \in E \cup \{\Delta\} \ \text{and} \ \lim_{k \to \infty} \sigma_{E \setminus F_k}(\omega) \ge \zeta (\omega) \}.
\]
\end{lemma}
Given the transition function $(P_t)_{t \ge 0}$  we can construct $\bM$ with continuous sample paths in $E_{\Delta}$ following the line of arguments in \cite[Section 4]{AKR} using in particular Lemma \ref{nestconti} and our further previous preparations (cf. \cite{RoShTr} for details).
\begin{thm}\label{existhunt}
There exists a Hunt process
\[
\bM =  (\Omega, \F, (\F_t)_{t \ge 0}, \zeta, (X_t)_{t \ge 0}, (\P_x)_{x \in E_{\Delta}}   )
\]
with state space $E$, having the transition function $(P_t)_{t \ge 0}$ as transition semigroup. In particular ${\bf (H2)}$ holds and $\bM$ satisfies the absolute continuity condition with respect to the part Dirichlet form $(\E^{E},D(\E^{E}))$, since by (\ref{partid})
\[
T_t^E f = P_t f \quad m\text{-a.e.} \ \forall t>0, \ f \in L^2(E,m)_b.
\]
Moreover, $\bM$ has continuous sample paths in the one point compactification $E_{\Delta}$ of $E$ with the cemetery $\Delta$ as point at infinity.
\end{thm}

\section{Pointwise weak existence of singular SDEs associated to Dirichlet forms}\label{s;3po}
Once having constructed (for instance through the steps $\bf{(H1)}$ and $\bf{(H2)}$) a Hunt process $\bM$ that satisfies the absolute continuity condition with respect to $(\E,D(\E))$, we want to identify the corresponding stochastic differential equation. 
We present two ways for the identification. The first is the well-known strict Fukushima decomposition. The second is realized by direct stochastic calculus.
\subsection{The strict Fukushima decomposition}\label{s;3.1sfdm}
This subsection refers to the monograph \cite{FOT}, hence some of its standard notations may be adopted below without definition. Throughout this subsection, we assume that $(\E,D(\E))$ is symmetric and that $\bf{(H1)}$ and $\bf{(H2)}$ hold.\\
In some cases, we will apply the strict Fukushima decomposition (i.e. Proposition \ref{p;LSFD} below) on an open subset $B \subset E$. Therefore, we need first to state some definitions and properties in a local setting. 
\begin{definition}
Let $B$ be an open set in $E$. For $x \in B, t \ge 0, \alpha>0$ and $p \in [1,\infty)$  let
\begin{itemize}
\item $\sigma_{B^c} := \inf\{t>0 \,|\;X_t \in B^c\}$, $D_{B^c} := \inf\{t\ge0 \,|\;X_t \in B^c\}$,
\item $P^{B}_{t}f(x) : = \EE_x [f(X_t) ; t<\sigma_{B^c} ]\; , \; f \in \mathcal{B}_{b}(B)$,
\item $R^{B}_{\alpha}f(x) : = \EE_x \Big[\int_{0}^{\sigma_{B^c}} e^{-\alpha s} f(X_s) \, ds \Big]  \; , \; f \in \mathcal{B}_{b}(B)$ ,
\item $D(\E^{B}): = \{u \in D(\E) \, | \; u=0 \,\,  \mathcal{E}\text{-q.e} \; on  \; B^c \}$.
\item $\E^{B} : = \mathcal{E} \, |_{D(\E^{B})\times D(\E^{B})}$.
\item $L^2(B \, ,m): = \{u \in L^2(E , m) \,|\; u=0, \; m \text{-a.e. on} \; B^c\}$.
\item $||f||^p_{p,B}: = \int_{B} |f|^p \; dm$. 
\item $||f||_{\infty,B}:= \inf \Big\{c>0 \,|\; \int_{B}  1_{\{ \,|f|>c \, \} } \, dm = 0 \Big\}$.
\item $\E^{B}_1(f,g) : = \E^{B}(f,g) + \int_{B} f g \; dm,  \;\;     f,g \in D(\E^{B})$.
\item $\| \,f\, \|_{D(\mathcal{E}^{B})} : = \E^{B}_1(f,f)^{1/2},  \;\;     f \in D(\E^{B})$.
\end{itemize}
\end{definition}
$(\E^{B},D(\E^{B}))$ is called the part Dirichlet form of $(\E, D(\E))$ on $B$. It is a regular Dirichlet form on $L^2(B, m)$ (cf. \cite[Section 4.4]{FOT}).
Let  $(T^{B}_t)_{t > 0}$ and $(G^{B}_{\alpha})_{\alpha > 0}$ be the $L^2(B, m)$-semigroup and resolvent associated to $(\E^{B},D(\E^{B}))$.
Then $P^{B}_{t} f, \;  R^{B}_{\alpha}f$ is an $m$-version of $T^{B}_t f,  G^{B}_{\alpha}f$, respectively for any $f \in L^2(B,m)_b$. Since $P_t^{B} 1_{A}(x) \le P_t 1_{A}(x)$ for any $ A\in \mathcal{B}(B)$, $x \in B$ and $m$ has full support on $E$, $A \mapsto P_t^{B} 1_{A}(x), \; A \in \mathcal{B}(B)$ is absolutely continuous with respect to $1_{B} \cdot m$. Hence there exists a (measurable) transition kernel density $p^{B}_{t}(x, y)$, $x,y \in B$, such that
\begin{equation}\label{abspart}
P_t^{B} f(x) = \int_{B} p_t^{B} (x,y) \,  f(y) \, m(dy) ,\; t>0 \;, \;\; x\in B
\end{equation}
for $f  \in  \mathcal{B}_{b}(B)$. Correspondingly, there exists a (measurable) resolvent kernel density $r_{\alpha}^{B}(x,y)$, such that
\[
R_{\alpha}^{B} f(x) = \int_{B} r_{\alpha}^{B} (x,y) \,  f(y) \, m(dy) \,, \;\; \alpha>0, \;\;x \in B
\]
for $f  \in \mathcal{B}_{b}(B)$.
For a signed Radon measure $\mu$ on $B$, let us define
\[
R_{\alpha}^{B} \mu (x) = \int_{B} r_{\alpha}^{B} (x,y) \, \mu(dy) \, , \;\; \alpha>0, \;\;x \in B
\]
whenever this makes sense.
The process defined by
\[
X^{B}_t(\omega)=
\begin{cases}
X_t(\omega), \;\;\;\; 0\le t < D_{B^c} (\omega) \\
\Delta, \;\;\;\; t \ge D_{B^c} (\omega)
\end{cases}
\]
is called the part process corresponding to $\E^{B}$ and is denoted by $\bM|_{B}$.  $\bM|_{B}$ is a Hunt process on $B$ (see \cite[p.174 and Theorem A.2.10]{FOT}).  In particular, by (\ref{abspart})
$\bM|_{B}$ satisfies the absolute continuity condition on $B$.\\ 
A positive Radon measure $\mu$ on $B$ is said to be of finite energy integral if
\[
\int_{B} |f(x)|\, \mu (dx) \leq C \sqrt{\E^{B}_1(f,f)}, \; f\in D(\E^{B}) \cap C_0(B),
\]
where $C$ is some constant independent of $f$. A positive Radon measure $\mu$ on $B$ is of finite energy integral (on $B$) if and only if there exists a unique function $U_{1}^{B} \, \mu\in D(\E^{B} )$ such that
\[
\E^{B}_{1}(U_{1}^{B} \, \mu, f) = \int_{B} f(x) \, \mu(dx),
\]
for all $f \in D(\E^{B}) \cap C_0(B)$. $U_{1}^{B} \, \mu$ is called $1$-potential of $\mu$. In particular, $R_{1}^{B} \mu$ is a version of $U_{1}^{B} \mu$ (see e.g. \cite[Exercise 4.2.2]{FOT}). The measures of finite energy integral are denoted by $S_0^{B}$. We further define the supersmooth measures $S_{00}^{B} : = \{\mu\in S_0^{B} \, | \; \mu(B)<\infty, \|U_{1}^{B} \mu\|_{\infty, B}<\infty \}$.
A positive Borel measure $\mu$ on $B$ is said to be smooth in the strict sense if there exists a sequence $(E_k)_{k\ge 1}$ of Borel sets increasing to $B$ such that $1_{E_{k}} \cdot \mu \in S_{00}^{B}$ for each $k$ and
\[
\P_{x} ( \lim_{k \rightarrow \infty} \sigma_{ B \setminus E_{k} }  \ge \zeta ) =1 \;, \;\; \forall x \in B.
\]
The totality of the smooth measures in the strict sense is denoted by $S_{1}^{B}$ (see \cite{FOT}). If  $\mu \in S_{1}^{B}$,
then there exists a unique $A \in A_{c,1}^{+, B}$ with $\mu = \mu_{A}$, i.e. $\mu$ is the Revuz measure of $A$ (see \cite[Theorem 5.1.7]{FOT}), such that
\[
\EE_x \Big[ \int_{0}^{\infty} e^{-t} \, d A_{t}  \Big] = R_1 \mu_{A}(x) \, , \;\; \forall x \in B.
\]
Here, $A_{c,1}^{+,B}$ denotes the positive continuous additive functionals on $B$ in the strict sense. If $B = E$, we omit the superscript $B$ and simply write $U_{1}, S_0, S_{00}, S_{1}$, and $A_{c,1}^{+}$.\\
For later purpose we state some auxiliary result (see \cite[Lemma 2.12]{ShTr13a}).
\begin{lemma}\label{l;sumadd}
For $k \in \mathbb{Z}$, let $ \mu_{A^{k}}, \mu_{A} \in S_{1}^{B}$ be the Revuz measures associated with $A^{k}, A \in A_{c,1}^{+, B}$, respectively.
Suppose that $ \mu_{A} = \sum_{k \in \mathbb{Z}}  \mu_{A^{k}}$. Then $A= \sum_{k \in \mathbb{Z}} A^{k}$.
\end{lemma}
Now we restate the strict Fukushima decomposition for continuous functions (cf. \cite[Theorem 5.5.5]{FOT}) which holds under our present assumptions on $(\E,D(\E))$:
\begin{proposition}\label{p;LSFD}
Suppose that a function $f$ satisfies the following conditions: \\
(i) $f \in D(\E)_{b,loc}$, $f$ is continuous on $E$\\
(ii) $1_G \cdot \mu_{\langle M^{[f]}\rangle} \in S_{00}$  for any relatively compact open set $G \subset E$\\
(iii) $\exists \nu = \nu^{(1)}-\nu^{(2)}$ with $1_{G} \cdot \nu^{(1)}, 1_{G} \cdot \nu^{(2)} \in S_{00}$ for any relatively compact open set $G \subset E$ and
\[
\E(f,g)= \int_E g\,d\nu, \quad  \forall g\in \mathscr{C},
\]
for some special standard core $\mathscr{C}$ of $\E$.

Let $A^{(1)}, A^{(2)}$ and $B$ be the positive continuous additive functional in the strict sense with Revuz measures $\nu^{(1)}, \nu^{(2)}$ and $\mu_{<M^{[f]}>}$, respectively.\\
Then,
\[
f(X_t)-f(X_0) = M_t^{[f]} + N_t^{[f]}, \quad \P_x\text{-a.s.} \quad \forall x \in E.
\]
Here,
\[
N^{[f]} = -A^{(1)} + A^{(2)}, \quad \P_x\text{-a.s.} \quad \forall x \in E,
\]
and $M^{[f]}$ is a local additive functional in the strict sense such that, for any relatively compact open set $G \in E$,
\begin{eqnarray*}
\mathbf{E}_x(M_{t \wedge \sigma_{G^c}}^{[f]})&=&0, \quad  \forall x \in G,\\
\mathbf{E}_x((M_{t \wedge \sigma_{G^c}}^{[f]})^2)&=& \mathbf{E}_x(B_{t \wedge \sigma_{G^c}}), \quad  \forall x \in G.
\end{eqnarray*}
\end{proposition}
Applied to concrete situations Proposition \ref{p;LSFD}(iii) will serve for the identification of the drift of the corresponding SDE, since $\nu$ can be interpreted as $- Lf \, dm$ at least if $f\in D(L)$, so that $N_t^{[f]}=\int_0^t Lf(X_s)ds$ since $h\,dm$ is uniquely associated to $(\int_0^t h(X_s)ds)_{t\ge 0}$ via the Revuz correspondence for any $h\in L^1_{loc}(E,m)$. But of course Proposition \ref{p;LSFD}(iii) identifies the drift also if it is not absolutely continuous with respect to  $ds$, for instance if it is a local time .\\
In order to handle quadratic variations and general drifts described through signed supersmooth measures by verifying Proposition \ref{p;LSFD}(ii) and (iii), we use the following key observation:
\begin{proposition}\label{p;smooth}
Let $\mu$ be a positive Radon measure on $E$,  $G \subset E$ some relatively compact open set, 
and $r_1^G \in C(E)$. Suppose that 
\begin{equation}\label{p;smoothb}
\int_G r_1(\cdot,y) \, \mu(dy) \le r_1^G
\end{equation}
$m$-a.e. on $E$ and that additionally at least one of the following conditions is satisfied:
\begin{itemize}
\item[(i)] $\int_G r_1(\cdot,y)\, \mu(dy)\in D(\E)$, i.e. $1_{G} \cdot \mu \in S_{0}$,
\item[(ii)] (\ref{p;smoothb}) holds $\mu$-a.e. on $G$,
\item[(iii)] $\int_G r_1(\cdot,y)\, \mu(dy)\in L^1(G,\mu)$, i.e.  $R_1 (1_{G} \cdot \mu) \in L^1(G,\mu)$. 
\end{itemize}
Then $1_{G} \cdot \mu \in S_{00}$. In particular, if this holds for any relatively compact open set $G$, then $\mu \in S_{1}$ with respect to a sequence of open sets $ (E_{k})_{k \ge 1}$.
\end{proposition}
\begin{proof}
Obviously (ii) implies (iii). The rest is just a reformulation of \cite[Proposition 2.13]{ShTr13a}.
\end{proof}
Next, we need to find a dominating continuous function $r_1^G$ as in (\ref{p;smoothb}). For this, let
\begin{equation}\label{rieszpot}
V_{\eta} g(x) : = \int_{\R^d} \frac{1}{\| x-y \|^{d-\eta}} \ g(y) \ dy, \quad x \in \R^d, \ \eta > 0,
\end{equation}
whenever it makes sense. The following result is from \cite[Chapter 4, Theorem 2.2]{Miz}.
\begin{lemma}\label{l;miz}
Let $\eta \in (0,d)$, $0<\eta - \frac{d}{p} <1$ and $g \in L^p(\R^d, dx)$ with  
\[
\int_{\R^d} (1+\|y\|)^{\eta -d} |g(y)| \ dy < \infty.
\]
Then $V_{\eta}g$ is H\"older continuous of order $\eta - \frac{d}{p}$.
\end{lemma}
In the case of existence of a nice resolvent kernel density estimate, we may find a continuous function $r_1^G$  as in Proposition \ref{p;smooth} by using Lemma \ref{l;miz}. Such a function is typically given as a linear combination of functions $V_{\eta} g$ (cf. for instance \cite[proof of Lemma 3.6(iii)]{ShTr13a}). However, in some cases such as Proposition \ref{p;3.8}(ii), Theorem \ref{t;3.9}(ii), Remark \ref{r;3.15}, Subsection \ref{s;3.113noref} and Subsection \ref{s;3122c}, the global resolvent density estimate (cf. e.g. Lemma \ref{l;Feller}(ii)) obtained from the heat kernel estimate in Theorem \ref{t;2.7she} is not sufficient for the application of Proposition \ref{p;smooth} via Lemma \ref{l;miz}. Or we simply do not know whether a global resolvent density estimate exists. In these cases, we use a method to obtain better Gaussian estimates for local resolvent kernel densities and combine it with a localization method. In this localization method Proposition \ref{p;LSFD} is applied locally on a nice exhaustion up to a capacity zero set of the state space which is typically Euclidean. Thus from now on up to the end of this subsection, we assume that $E\subset \R^d$.\\

Let $\varphi>0$ $dx$-a.e. on $E \subset \R^d$, $\varphi \in L^1_{loc}(E,dx)$ and $A=(a_{ij})_{1 \le i,j \le d}$ be a symmetric $d \times d$ matrix such that $a_{ij} \in L^1_{loc}(E, m)$ with $m:= \varphi dx$ and for $dx$-a.e. $x \in E$
\[
0 \le \langle A(x) \xi, \xi \rangle, \quad  \forall \xi \in \R^d.
\]
Suppose $(\E,D(\E))$ is conservative and given as the closure in $L^2(E,m)$ of $\frac{1}{2} \int_{E} \langle A \nabla f , \nabla g \rangle \,dm$, $f,g \in C^{\infty}_0(E)$, where $E$ is either closed or open. We assume that\\ 
\begin{itemize}
\item[{\bf (L)}] There exists an increasing sequence of bounded relatively open Lipschitz domains  $\{B_k\}_{ k \in \N} \subset E$,
with Cap$(E\setminus B)=0$, $B:=\cup_{k\ge 1}B_k$ and for any $k \ge 1$ there exists some constant $\kappa_{B_k} \ge 1$ such that  
$$
\kappa_{B_k}^{-1} \| \xi \|^2 \le  \langle A(x) \xi, \xi \rangle\le \kappa_{B_k} \| \xi \|^2, \forall \xi \in \R^d,  dx\text{-a.e. } x\in B_k, 
$$
and $\varphi(x)\in (d_k^{-1},d_k)$ for $dx$-a.e. $ x\in B_k$, where $d_k\nearrow \infty$ as $k\to \infty$.\\ 
\end{itemize}
Note that {\bf (L)} implies that 
$\E^{\overline{B}_k}(f,g) = \frac{1}{2} \int_{B_k} \langle A  \nabla f , \nabla g \rangle \,dm, \quad f,g \in C^{\infty}(\overline{B}_k)$, is closable in $L^2(\overline{B}_k,m) \equiv L^2(B_k,m)$, $k \ge 1$. We denote the closure by $(\E^{\overline{B}_k},D(\E^{\overline{B}_k}))$. \\
Given the strict global ellipticity of $A=(a_{ij})$ on each $B_k$ and that $\varphi$ is bounded above and below away from zero by a strictly positive constant on each $B_k$, we obtain exactly as in \cite[Lemma 5.4]{ShTr13a} the following lemma.
\begin{lemma}\label{T;NI1}(Nash type inequality)
Under {\bf (L)} it holds for any  $k \ge 1$:
\begin{itemize}
\item[(i)] If $d \ge 3$, then for $f \in D(\E^{\overline{B}_k})$
\begin{equation*}
\left\|f\right\|_{2,B_k}^{2 + \frac{4}{d}}\leq c_k \left[\E^{\overline{B}_k}(f,f) + \left\|f\right\|_{2,B_k}^2 \right]\left\|f\right\|_{1,B_k}^{\frac{4}{d}}.
\end{equation*}
\item[(ii)] If $d=2$, then for $f \in D(\E^{\overline{B}_k})$ and any $\delta>0$
\begin{equation*}
\left\|f\right\|_{2,B_k}^{2 + \frac{4}{d+\delta}} \le c_k \left[\E^{\overline{B}_k}(f,f) + \left\|f\right\|_{2,B_k}^2 \right]\left\|f\right\|_{1,B_k}^{\frac{4}{d+\delta}}.
\end{equation*}
Here $c_k >0 $ is a constant which goes to infinity as $k \rightarrow \infty$.
\end{itemize}
\end{lemma}
By \cite[(3.25)]{CKS}, the Nash type inequalities imply (upper) Gaussian heat kernel estimates for the heat kernel $p^{\overline{B}_k}_{t}(x,y)$ associated to $(\E^{\overline{B}_k},D(\E^{\overline{B}_k}))$ which exists uniquely for $m$-a.e. $x,y \in \overline{B_k}$ (cf. \cite[Proposition 5.5]{ShTr13a}). Since $(\E^{B_k},D(\E^{B_k}))$ is the part Dirichlet form of $(\E^{\overline{B}_k},D(\E^{\overline{B}_k}))$, it is easy to see that
\[
p^{B_k}_{t}(x,y) \le p^{\overline{B}_k}_{t}(x,y) \quad \text{for} \ \, m\text{-a.e.} \ x,y \in B_k.
\]
Thus the heat kernel estimate also holds for $p^{B_k}_{t}(x,y)$. By taking the Laplace transform of $p_{\cdot}^{B_k} (x,y)$ and using the heat kernel estimate, we obtain the following resovent kernel density estimates (cf. \cite[Corollary 5.6]{ShTr13a}).
\begin{cor}\label{C;RDE1}
Under {\bf (L)} we have for $m$-a.e. $x, y \in B_k$
\begin{itemize}
\item[(i)] if $d \ge 3$, then
\[
r^{B_k}_{1} (x,y)  \le \frac{C_k}{\|x-y\|^{d-2}}.
\]
\item[(ii)] if $d=2$, then for any $\delta>0$
\[
r^{B_k}_{1} (x,y)  \le \frac{C_k}{\|x-y\|^{d+\delta-2}}.
\]
Here $C_k >0 $ is a constant which goes to infinity as $k \rightarrow \infty$.
\end{itemize}
\end{cor}
Note that  the  part Dirichlet forms 
$(\E^{B_k},D(\E^{B_k}))$ of $(\E,D(\E))$ on $B_k$, $k\ge 1$, as well as the part Dirichlet form 
$(\E^{B},D(\E^{B}))$ of $(\E,D(\E))$ on $B$, inherit the properties $\bf{(H1)}$ and $\bf{(H2)}$ from $(\E,D(\E))$ by considering the part processes. Moreover, since $(\E,D(\E))$ is conservative and  Cap$(E\setminus B)=0$, we can use (\ref{partid}) to see that its part Dirichlet form $(\E^{B},D(\E^{B}))$ on $B$ is also conservative. In particular $\P_x \big(\sigma_{B^c}= D_{B^c} =\infty) =1$ for any $x\in B$ and exactly as in \cite[Lemma 5.10]{ShTr13a}, we show:
\begin{lemma}\label{l;limit}
$\P_x \big(\lim_{k \rightarrow \infty} D_{B_{k}^c} =\lim_{k \rightarrow \infty} \sigma_{B_{k}^c} =\infty \big)=1$ for all $x \in B$.
\end{lemma}
We may then apply Proposition \ref{p;LSFD} to the part Dirichlet forms $(\E^{B_k},D(\E^{B_k}))$ by using the resolvent kernel density estimate of Corollary \ref{C;RDE1}, Proposition  \ref{p;smooth} and Lemma \ref{l;miz}. Suppose that 
this is possible and that for each $k\ge 1$, we get according to Proposition \ref{p;LSFD} for the coordinate projections $f^j$, $1\le j\le d$ (obviously continuous and in $D(\E^{B_k})_{b,loc}$ for any $k\ge 1$)
\begin{equation}\label{local1}
X_{t}^j = x_j + M_{t}^{[f^j],k} +  N^{[f^j],k}_t,\quad t < \sigma_{B_k^c},\  \ \P_x\text{-a.s. for all } x\in B_k,
\end{equation}
where $N^{[f^j],k} = -A^{(1),k} + A^{(2),k}$ and $A^{(1),k}, A^{(2),k}$  are the positive continuous additive functionals in the strict sense of $\bM|_{B_k}$ with Revuz measures $1_{B_k}\cdot \nu^{(1)}, 1_{B_k}\cdot \nu^{(2)}\in S_{00}^{B_k}$, 
$\nu^{(1)}, \nu^{(2)}$ being smooth measures on $B$ with respect to $(\E^{B},D(\E^{B}))$, $M^{[f^j],k}$ is a MAF in the strict sense of $\bM|_{B_k}$ with covariations $\langle M^{[f^i],k}, M^{[f^j],k}\rangle _{t\wedge \sigma_{B_k^c}}=\int_0^{t \wedge \sigma_{B_k^c}} a_{ij}(X_s)ds$, $1\le i,j\le d$. 
Suppose further that we can show consistency, in the sense that 
$A^{(i),k}_t =  A^{(i),k+1}_t, \; \forall t < \sigma_{B_{k}^{c}}$ $\P_{x}$-a.s. for all $x \in B_{k}$, $i=1,2$, and  
$M_{t}^{[f^j],k} =  M_{t}^{[f^j],k+1}, \; \forall t < \sigma_{B_{k}^{c}}$ $\P_{x}$-a.s. for all $x \in B_{k}$. Then $M_{t}^{[f^j]} :=\lim_{k\to \infty}  M_{t}^{[f^j],k}$ is a well-defined local MAF in the strict sense of $\bM|_{B}$ and $A^{(i)}_t := \lim_{k\to \infty} A^{(i),k}_t$,  $i=1,2$, are well-defined positive continuous additive functionals in the strict sense of  $\bM|_{B}$ with Revuz measures  $\nu^{(1)}, \nu^{(2)}$.  By letting $k\to \infty$ in (\ref{local1})
\begin{equation*}
X_{t}^j = x_j + M_{t}^{[f^j]} +  N^{[f^j]}_t,\quad t \ge 0,\  \ \P_x\text{-a.s. for all } x\in B,
\end{equation*}
with $N^{[f^j]}=-A^{(1)} + A^{(2)}$ and $\langle M^{[f^i]},M^{[f^j]}\rangle =\int_0^{\cdot} a_{ij}(X_s)ds$. \\
We will refer to this 
as {\it localization procedure}. For explicit examples where it is performed in detail, we refer to \cite[Section 5]{ShTr13a} and \cite[Section 3.2]{ShTr15} and also here below (see Proposition \ref{p;3.8}(ii), Theorem \ref{t;3.9}(ii), Remark \ref{r;3.15}, Subsection \ref{s;3.113noref} and Subsection \ref{s;3122c}), where we indicate at least why {\bf (L)} is satisfied. In the examples below, $A^{(1),k}, A^{(2),k}$, $k\ge 1$, that appear in the localization procedure are sometimes given as infinite sums of strict PCAF's and we additionally have to make use of Lemma \ref{l;sumadd}.

\subsubsection{Symmetric distorted Brownian motion}\label{sdbm1}
 We assume throughout this subsection that $E = \R^d$, with $d \ge 3$ 
(except in Lemma \ref{l;Feller}(vi),  Proposition \ref{p;3.8}(ii), and Remark \ref{r;3.15} where the state space is $\R^d\setminus \{0\}$ with $d\ge 2$). 
We consider a weight function $\psi$ in the Muckenhoupt $A_2$ class, in notation $\psi\in A_2$ (cf. \cite{Tu}). 
Precisely, we assume the following:\\
\begin{itemize}
\item[{\bf ($\alpha$)}] $\phi :\R^{d} \rightarrow [0,\infty)$  is a $\mathcal{B}(\R^d)$-measurable function and $\phi > 0$ $dx$-a.e.,
\item[{\bf ($\beta$)}] $\rho \phi \in  A_{2}, \;\;  \rho \in H^{1,1}_{loc}(\R^d, dx)$, $\rho > 0$ $dx$-a.e.,\\
\end{itemize}
and consider 
\begin{equation}\label{DF}
\E (f,g) : = \frac{1}{2}\int_{\R^d}\nabla f\cdot\nabla g \   dm, \quad f,g \in C_0^{\infty}(\R^d), \quad m := \rho \phi dx
\end{equation}
in $L^2(\R^d, m)$.

\begin{remark}\label{r;re1}
Let $\tilde{c} \ge 1$. If $\phi$ is measurable with $\tilde{c}^{-1} \le \phi \le \tilde{c}$ and $\rho \in A_{2}$, then $\rho \phi \in A_{2}$.
\end{remark}

Since $\rho \phi \in A_2$, we have $\frac{1}{\rho \phi}  \in L^1_{loc}(\R^d, dx)$, and the latter implies that  $\eqref{DF}$  is closable in $L^2(\R^d, m)$ (see \cite[II.2 a)]{MR}). The closure $(\E,D(\E))$ of $\eqref{DF}$ is a strongly local, regular, conservative, symmetric Dirichlet form (cf. e.g. \cite[p. 274]{St3}). Furthermore  $(\E,D(\E))$ satisfies properties (i)-(iv) of Definition \ref{p;dcpo} (cf. \cite[5.B]{St3}).  Therefore there exists a jointly continuous transition kernel density  $p_{t}(x,y)$ as stated in \eqref{e;tdst3} with $E=\R^d$. Note that in our case $\gamma (x,y) = \|x- y \|$, $x,y \in \R^d$. 
Moreover, Proposition \ref{p;strongfl} applies, so that $(P_t)_{t \ge 0}$ is in particular strong Feller. Next, we assume that\\
\begin{itemize}
\item[{\bf ($\gamma$)}] the transition function $(P_t)_{t \ge 0}$ satisfies (\textbf{H2}) with $E= \R^d$. \\
\end{itemize}
We will use the Feller method and the Dirichlet form method for some typical Muckenhoupt $A_2$ weights to verify {\bf ($\gamma$)} later.  Since $\rho \phi \in A_2$,
$(\E,D(\E))$ is conservative, i.e. $T_t 1(x)=1$ for $m$-a.e. $x\in \mathbb{R}^d$ and all $t>0$ (see \cite[Proposition 2.4]{ShTr13b}). It follows
\[
\P_{x}(\zeta = \infty)=1, \ \ \forall x \in \R^d,
\]
by \cite[Theorem 4.5.4(iv)]{FOT} (or Remark \ref{r;connoex} since $(P_t)_{t \ge 0}$ is strong Feller) and 
\[
\P_x \big(t \mapsto X_{t} \text{ is continuous on} \  [0, \infty)\big) =1, \ \ \forall x \in \R^d,
\]
by \cite[Theorem 4.5.4(ii)]{FOT}. In order to be explicit, we further assume the following integration by parts formula\\

\begin{itemize}
\item[{\bf (IBP)}]  for $ f  \in \{f^1, \dots,f^d\}$, $g \in C_{0}^{\infty}(\R^{d})$
\[
-\E (f , g) = \int_{\R^d} \left (\nabla f  \cdot  \frac{\nabla \rho  }{2\rho}  \right ) g \, dm
+ \int_{\R^{d}} g \, d\nu^{f},
\]
where $\nu^{f}= \sum_{k \in \mathbb{Z}} \nu_{k}^{f}$ and $\nu^{f},  \nu_{k}^{f}, k \in \mathbb{Z}$ are signed Radon measures (locally of bounded total variation).
\end{itemize}

For a signed Radon measure $\mu$ we denote by $\mu^+$ and $\mu^-$ the positive and negative parts in the Hahn decomposition for $\mu$, i.e. $\mu = \mu^+ -\mu^-$.
Additionally, we assume that\\
\begin{itemize}
\item[{\bf ($\delta$)}]  for any $G \subset \R^d$ relatively compact open, $k \in \mathbb{Z}$ and $f \in \{f^1, \dots, f^d\}$, we have that $1_G \cdot \nu^{f+}$, $1_G \cdot  \nu^{f-}$, $1_G  \cdot  \nu_k^{f+}$, $1_G \cdot  \nu_k^{f-}, 1_G \cdot  \frac{\|\nabla \rho\|}{\rho} m \in S_0$ and the corresponding 1-potentials are all bounded by continuous functions.
\end{itemize}

\begin{thm}\label{t;Ldec}
Suppose {\bf $(\alpha)-(\delta)$} and {\bf (IBP)}. Then
\begin{equation}\label{LSFD1}
X_{t} = x + W_{t} + \int_{0}^{t}\frac{\nabla \rho}{2 \,\rho}(X_s) \, ds + \sum_{k \in \mathbb{Z} } L_{t}^{k} \;,\;\; t  \ge 0,
\end{equation}
$\P_x$-a.s. for any $x \in \R^{d}$ where $W$ is a standard d-dimensional Brownian motion starting from zero, $L^{k} = (L^{1, k}, \dots, L^{d, k})$ and $L^{j,k}, j=1,\dots, d$, is the difference of positive continuous additive functionals of $X$  in the strict sense associated with Revuz measure $\nu_{k}^{f^{j}} = \nu_{k}^{f^{j},(1)} - \nu_{k}^{f^{j},(2)}$ defined in (IBP) (cf. \cite[Theorem 5.1.3]{FOT}).
\end{thm}
\begin{proof}
Given that $(\alpha)-(\delta)$ and (IBP) hold, the assertion follows from \cite[Theorem 5.1.3]{FOT}, Lemma \ref{l;sumadd}, and Propositions \ref{p;LSFD}, \ref{p;strongfl} and \ref{p;smooth}.
\end{proof}
\begin{remark}\label{remhkenonex}
The heat kernel estimate (\ref{e;hkeos}) is not explicit, since the volumes of the $m$-balls in it are unknown. Therefore its use is in a sense limited. While it was possible to obtain already good information about the transition function in Proposition \ref{p;strongfl}, the last ingredient to obtain \textbf{(H2)}$^{\prime}$(i), (ii) or the Feller property of the transition function is missing. Assuming an explicit estimate on the weight of $m$ is the main additional ingredient for the proof of the following lemma (cf. \cite[Lemma 3.6]{ShTr13a}). For other weights the proof of the lemma can serve as a toy model to show how the full information can be obtained.
\end{remark}
\begin{lemma}\label{l;Feller}
Let $\tilde{c}^{-1} \|x\|^{\alpha} \le \rho \phi(x) \le \tilde{c} \|x\|^{\alpha}$ for some $\alpha \in (-d,d)$, $\tilde{c} \ge 1$.
Then {\bf $(\alpha)$} and {\bf $(\beta)$} are satisfied and
\begin{itemize}
\item[(i)] $\lim_{t \downarrow 0}P_t f(x) = f(x)$, $\forall x \in \R^d$, $\forall f \in C_0(\R^d)$, i.e. (\textbf{H1}) and (\textbf{H2}) hold (cf. Proposition \ref{p;strongfl}(i),(iii) and Lemma \ref{t;Feller}) and $(P_t)_{t\ge 0}$ is a Feller semigroup.
\item[(ii)] Let $\Phi(x,y): = \frac{1}{\|x-y\|^{\alpha +d -2}}$ and $\Psi(x,y):= \frac{1}{\|x-y\|^{d -2} \|y\|^{\alpha}}$. Then 
\[
c^{-1} \left(\Phi(x,y) + \Psi(x,y)1_{\{ \alpha \in [0,d)\}} \right)  \le r_1(x,y) \le c \left( \Phi(x,y) + \Psi(x,y)1_{\{ \alpha \in (-d,0)\}}  \right).
\]
\item[(iii)] Let $\alpha \in (-d+1,2) $ and $G\subset \R^d$ any relatively compact open set. Suppose $1_G  \cdot f \ \| x\|^{\alpha} \in L^p(\R^d,dx)$, $p \ge 1$ with $0 < 2- \alpha - \frac{d}{p} <1$ and $1_G \cdot  f \in L^q(\R^d,dx)$ with $0 < 2- \frac{d}{q} <1$. Then $R_1(1_G \cdot  |f| m)$ is bounded everywhere (hence clearly also bounded $m$-a.e. on $\R^d$ and $R_1(1_G |f| m) \in L^1(G,|f|m)$) by the continuous function 
$\int_G |f(y)| \ \left(\Phi(\cdot,y) + \Psi(\cdot,y) \right) \, m(dy)$. In particular, Proposition \ref{p;smooth} applies and $1_G  \cdot  |f| m \in S_{00}$.
\item[(iv)] Let $\alpha \in (-d+1,2)$. Then $R_1\left(1_G \cdot  \frac{\|\nabla \rho\|}{\rho} m \right)$ is pointwise bounded by a continuous function for any relatively compact open set $G \subset \R^d$. In particular $1_G \cdot  \frac{\|\nabla \rho\|}{\rho} m \in S_{00}$ for any relatively compact open set $G \subset \R^d$.
\item[(v)] Let  $\alpha \in (-d+1,1)$. Let $D \subset \R^d$ be a bounded Lipschitz domain with surface measure $\sigma_{\partial D}$. Suppose that $\rho$ is bounded on $\partial D$ (more precisely the trace of $\rho$ on $\partial D$, which exists since $\rho \in H^{1,1}_{loc}(\R^d)$). Then $R_1(1_G \cdot \rho  \sigma_{\partial D})$ is pointwise bounded by a continuous function for any relatively compact open set $G \subset \R^d$. In particular $1_{G} \cdot \rho  \sigma_{\partial D} \in S_{00}$ for any relatively compact open $G \subset \R^d$.
\item[(vi)] Let $\alpha \in [-d+2,d)$, $d\ge 2$. Then $\emph{Cap}(\{0\})=0$ and the part Dirichlet form $(\E^B,D(\E^B))$ on $B:=\R^d \setminus \{0\}$ satisfies (\textbf{H1}), (\textbf{H2}) with transition kernel density $p_t^B = p_t |_{B \times B}$. Moreover $(\E^B,D(\E^B))$ is conservative.
\end{itemize}
\end{lemma}

\paragraph{Skew reflection on spheres}
Let $m_0 \in (0, \infty)$ and $(l_{k})_{k \in \mathbb{Z}} \subset (0, m_0)$, $0 < l_{k} < l_{k+1} <m_0$, be a sequence converging to $0$ as $k \rightarrow -\infty$ and converging to $m_0$ as $k \rightarrow \infty$, $(r_{k})_{k \in \mathbb{Z}} \subset (m_0, \infty)$, $m_0 < r_{k} < r_{k+1} < \infty$, be a sequence converging to $m_0$ as $k \rightarrow -\infty$ and tending to infinity  as $k \rightarrow \infty$, and set
\begin{equation}\label{eqphi}
\phi : = \sum_{k \in \mathbb{Z} } \left (\gamma_{k}  \cdot 1_{A_{k}}
+ \overline{\gamma}_{k} \cdot 1_{\hat{A}_{k}}\right ),
\end{equation}
where $\gamma_{k}$ , $ \overline{\gamma}_{k} \in (0, \infty)$, $A_{k} : = B_{l_{k}} \setminus \overline{B}_{l_{k-1}}$ , $\hat{A}_{k} : = B_{r_{k}} \setminus \overline{B}_{r_{k-1}}$, $k\in \mathbb{Z}$. 
\begin{prop}\label{p;3.8}
 Let $\phi$ be as in \eqref{eqphi}. Suppose that
\begin{equation*}
\sum_{k \in \mathbb{Z}} |\,  \gamma_{k+1} - \gamma_{k} \,| + \sum_{k \le 0} |\,  \overline{\gamma}_{k+1} - \overline{\gamma}_{k} \,| < \infty \ \  \text{ and } \ \ \tilde{c}^{-1} \le \phi \le \tilde{c}\ \text{ for some }\tilde{c} \ge 1.
\end{equation*}
\begin{itemize}
\item[(i)]Let $\rho(x) = \|x\|^{\alpha}$, $\alpha \in (-d+1,2)$. If $\phi \equiv \tilde{c}$ $dx$-a.e. 
(i.e. below in (\ref{equast}) and (\ref{etarep}) it holds $\eta \equiv 0$) or $\phi \nequiv \tilde{c}$ $dx$-a.e. and $\alpha \in (-d+1,1)$. Then the processes  $\big ( (X_t)_{t\ge 0}, \P_x \big )$ and $\big ( (\|X_t\|)_{t\ge 0}, \P_x \big )$ are continuous semimartingales and 
\begin{equation}\label{equast}
X_t = x + W_t + \frac{\alpha}{2}\int^{t}_{0} X_s \|X_s\|^{-2} \,ds +  \int_0^{\infty}\int_0^t \nu_{a}(X_s) \, d\ell_s^a(\|X\|)\,\eta(da), \ t\ge 0, \ \P_x \text{-a.s.} 
\end{equation}
for any $x \in \R^d$, where $W$ is a standard d-dimensional Brownian motion starting from zero, $\nu_{a} = (\nu_{a}^1 , \dots , \nu_{a}^d)$, $a>0$ is the unit outward normal vector on $\partial  B_{a}$, 
$\ell_{t}^{a}(\|X\|)$ is the symmetric semimartingale local time of $\|X\|$ at $a\in (0,\infty)$ as defined in \cite[VI.(1.25)]{RYor},  $\gamma : = \lim_{k \rightarrow \infty} \gamma_{k}$, $\overline{\gamma} : = \lim_{k \rightarrow -\infty} \overline{\gamma}_{k}$ and
\begin{equation}\label{etarep}
\eta : =  \sum_{k \in \mathbb{Z}} \left ( \frac{ \gamma_{k+1} - \gamma_{k}  } {  \gamma_{k+1} + \gamma_{k}}  \,\delta_{l_{k}}+ \frac{  \overline{\gamma}_{k+1} - \overline{\gamma}_{k} }{  \overline{\gamma}_{k+1} + \overline{\gamma}_{k} } \,\delta_{r_{k}}\right )+ \; \frac{ \overline{\gamma}  - \gamma }{ \overline{\gamma}  + \gamma }\,\delta_{m_0}.
\end{equation}
\item[(ii)] Let $\rho(x) = \|x\|^{\alpha}$, $\alpha \in [1,d)$, $d\ge 2$.  Then \eqref{equast} holds for any $x \in \R^d \setminus \{0\}$. 
\end{itemize}
\end{prop}
\begin{proof} 
(i) {\bf $(\alpha)$}, {\bf $(\beta)$} hold by Remark \ref{r;re1} since $\rho(x) = \|x\|^{\alpha}$, $\alpha \in (-d,d)$ is an $A_2$-weight (see \cite[Example 1.2.5]{Tu}). {\bf $(\gamma)$} follows from Lemma \ref{l;Feller}(i) by the Feller method, i.e. the corresponding transition semigroup is Feller. {\bf (IBP)} follows as in \cite[Proposition 3.1]{ShTr13b} and {\bf $(\delta)$} follows from Lemma \ref{l;Feller}(iv) and (v). Thus Theorem \ref{t;Ldec} applies. The identification of the drift part in (\ref{LSFD1}) as sum of local times then follows as in 
\cite[Section 5]{ShTr13b} using the integration by parts formula \cite[Proposition 3.1]{ShTr13b}. \\
(ii) By Lemma \ref{l;Feller}(vi) $(\E^B,D(\E^B))$, $B:= \R^d \setminus \{0\}$ satisfies (\textbf{H1}), (\textbf{H2}), is conservative, and Cap$(\{0\})=0$. Fix $\alpha \in [1,d)$. Let 
\[
B_k : = \Big\{x \in \R^d \ \Big| \ \frac{l_{-k+1} + l_{-k}}{2} < \|x\| < \frac{r_{k+1} + r_k}{2} \Big\}, \quad k \ge 1.
\]
Then
\[
b_k := \tilde{c}^{-1} \Big(\frac{l_{-k+1} + l_{-k}}{2}\Big)^{\alpha} < \rho \phi  < \tilde{c} \ \Big(\frac{r_{k+1} + r_k}{2}\Big)^{\alpha} = : e_k
\]
on $B_k$. Set $d_k := \max (b_k^{-1}, e_k)$, $k \ge 1$. Then $(B_k)_{k \ge 1}$ is an increasing sequence of relatively compact open sets with smooth boundary such that $\bigcup_{k \ge 1} B_k  = B$ and $ \rho \phi \in (d_k^{-1},d_k)$ on $B_k$ where $d_k \to \infty$ as $k \to \infty$. Moreover $\|\nabla \rho \| \in L^{\infty}(B_k,dx)$ for any $k \ge 1$. We can now apply the localization procedure as explained after Lemma \ref{l;miz}, since one easily verifies that condition {\bf (L)} is satisfied. We only repeat here once again that the Nash type inequality of Lemma \ref{T;NI1} allows for local resolvent kernel density estimates as in Corollary \ref{C;RDE1} and these local estimates are usable in contrast to the global ones of Lemma \ref{l;Feller}(ii), which do not lead to any result.
\end{proof}
\begin{remark}
For an interpretation of the drift part in (\ref{equast}), we refer to \cite[Remark 2.7]{ShTr13b}.
\end{remark}

\paragraph{Skew reflection on a Lipschitz domain} Let 
\begin{equation}\label{eqphi2}
\phi(x) : = \beta \, 1_{G^{c}}(x) + (1- \beta) 1_{G} (x), \ \ \ \rho(x) : = \|x\|^{\alpha}, \ \alpha \in (-d+1,d),
\end{equation}
where $\beta \in (0,1)$ and $G \subset \R^d$ is a bounded Lipschitz domain. Consider the Dirichlet form determined by \eqref{DF} with $\phi$ and $\rho$ as in (\ref{eqphi2}). Then the following integration by parts formula holds for $f\in \{f^1,\dots,f^d\}$, $g \in C_0^{\infty}(\R^d)$
\begin{eqnarray}\label{eqphi2b}
-\E(f,g) &= &\int_{\R^d} \left(\nabla f \cdot \frac{\nabla \rho}{2 \rho} \right) \ g \, dm + (2\beta -1) \int_{\partial G} \nabla f \cdot \nu \ \frac{\rho } {2} \ d \sigma,
\end{eqnarray}
where $\nu$ denotes the unit outward normal on $\partial G$ (cf. \cite{Tr1} and \cite{Tr3}). The existence of a Hunt process associated to $\E$ that satisfies the absolute continuity condition follows from Lemma \ref{l;Feller} (i). Furthermore:
\begin{thm}\label{t;3.9} Let $\phi$, $\rho$ be as in (\ref{eqphi2}). Then we have:
\begin{itemize}
\item[(i)] Let $\alpha \in (-d+1,1)$. Then
\begin{equation}\label{SFD1} 
X_{t} = x + W_{t} + \frac{\alpha}{2} \int_{0}^{t} X_s \|X_s\|^{-2} \, ds + (2 \beta -1) \int_0^t  \nu (X_s)\, d \ell_s \quad t  \ge 0
\end{equation}
$\P_x$-a.s. for all $x \in \R^d$, where $(W_t)_{t \ge 0}$ is a d-dimensional Brownian motion starting from zero and $(\ell_t)_{t\ge 0} \in A^{+}_{c,1}$ is uniquely associated to the surface measure $\frac{\rho \sigma}{2}$ on $\partial G$ via the Revuz correspondence.
\item[(ii)] Let $0 \notin \partial G$ and $\alpha \in [1,d)$, $d\ge 2$. Then \eqref{SFD1} holds $\P_x$-a.s. for any $x \in \R^d \setminus \{0\}$.
\end{itemize}
\end{thm}
\begin{proof} 
(i)  Exactly as in the proof of  Proposition \ref{p;3.8}(i) $(\alpha)$-$(\delta)$  are satisfied and {\bf (IBP)} holds by (\ref{eqphi2b}). Then the assertion immediately follows from Theorem \ref{t;Ldec}.\\
(ii) Fix $\alpha \in [1,d)$. We have either
$0 \in G$ or $0 \in \overline{G}^c$. If $0 \in G$, then choose $k_0 \ge 1$ such that $\partial G \subset \{x \in \R^d \ | \ k_0^{-1} < \|x\| < k_0\}$ and let
\[
B_k := \{x \in \R^d \ | \ (k_0 + k)^{-1} < \|x\| < k_0 + k \}, \quad  k \ge 1.
\]
Then
\[
b_k : = \min(\beta, 1-\beta)(k_0 + k)^{-\alpha} <  \rho \phi < \max(\beta,1-\beta)(k_0 + k)^{\alpha} = : e_k
\]
and we let $d_k : = \max(b_k^{-1},e_k)$, $k \ge 1$. (If $0 \in \overline{G}^c$ then similarly, we can find suitable $(B_k)_{k \ge 1}$ and $(d_k)_{k \ge 1}$.)  One easily checks that assumption {\bf (L)} is satisfied with respect to the sequences $(B_k)_{k \ge 1}$ and $(d_k)_{k \ge 1}$. Then we proceed as in the proof of Proposition \ref{p;3.8}(ii).
\end{proof} 
\begin{remark}
Theorem \ref{t;3.9} extends a result obtained by Portenko in \cite[III, \S 3 and \S 4]{port}.
\end{remark}

\paragraph{Skew reflection on hyperplanes}
Let  $(l_{k})_{k \in \mathbb{Z}} \subset (-\infty, 0)$, $- \infty < l_{k} < l_{k+1} < 0$  be a sequence converging to $0$ as $k \rightarrow \infty$ and tending to $- \infty$ as $k \rightarrow -\infty$. Let $(r_{k})_{k \in \mathbb{Z}} \subset (0, \infty)$, $0 < r_{k} < r_{k+1} < \infty$ be a sequence converging to $0$ as $k \rightarrow -\infty$ and tending to infinity  as $k \rightarrow \infty$. Set
\begin{equation}\label{eqphi3}
\phi(x_{d}):= \sum_{k \in \mathbb{Z} }  \Big(  \gamma_{k+1}  \cdot 1_{(l_{k}, l_{k+1} )}(x_{d})+ \overline{\gamma}_{k+1} \cdot 1_{(r_{k}, r_{k+1} )}(x_{d}) \Big)
\end{equation}
where $\gamma_{k}, \overline{\gamma}_{k} \in (0, \infty)$. Note that $\phi$ only depends on the $d$-th coordinate 
and that $\phi$ has discontinuities along the hyperplanes
\[
H_{s} := \{ x \in \R^d \,|\ x_d =s \}, \;\; s \in \{0,l_k,r_k;k\in\mathbb{Z}\}.
\]
Consider the assumptions
\begin{itemize}
\item[{\bf (a)}] $\rho \phi \in A_{2}$ and $\rho(x) = \| x\|^{\alpha}$, $\alpha \in (-d+1,1)$,
\item[{\bf (b)}] $\sum_{k \ge 0} |\,  \gamma_{k+1} - \gamma_{k} \,| + \sum_{k \le 0} |\,  \overline{\gamma}_{k+1} - \overline{\gamma}_{k} \,| < \infty$
and $\gamma: = \lim_{k \to \infty} \gamma_k$, $\overline{\gamma}: = \lim_{k \to-\infty} \overline{\gamma}_k$ are strictly positive. 
\end{itemize}
\begin{proposition}\label{t;eprocess}
Let $\phi$ be as in (\ref{eqphi3}) and assume that {\bf (a)}, {\bf (b)} hold. Then {\bf ($\alpha$)}-({\bf $\gamma$}) hold.
\end{proposition}
\begin{proof} 
The assumptions {\bf (a)}, {\bf (b)} imply {\bf ($\alpha$)}, {\bf ($\beta$)}. Therefore, the closure ($\E,D(\E)$) of \eqref{DF}
is a symmetric, regular and strongly local Dirichlet form.  Using the integration by parts formula \cite[Proposition 3.11]{ShTr13a} one can see that the functions $f \in C_0^{\infty}(\R^d)$ satisfying 
\begin{eqnarray}\label{domgenerator}
&& \partial_d f(\bar{x},l_k) = \partial_d f(\bar{x},r_k) =\partial_d f(\bar{x},0) = 0 \ \ \text{for all} \ k \in \mathbb{Z} \notag \\
\text{and}  && \frac{1}{2} \Delta f  + \nabla f \, \cdot \frac{\nabla \rho  }{2\rho} \in L^2(\R^d, m) 
\end{eqnarray}
are in $D(L)$ where $\bar{x} = (x_1,\dots,x_{d-1}) \in \R^{d-1}$.
For given $r \in (0,\infty)$, define $S_r$ to be the set of functions $h \in C_0^{\infty}(\R^d)$ such that 
\begin{equation}\label{scond}
\nabla h(x) = 0, \ \ \forall x \in B_r, \quad \quad \partial_d h (\bar{x},x_d) = 0 \ \ \text{if} \ -r <  x_d < r
\end{equation}
and $h$ satisfies \eqref{domgenerator}. Note that if $h \in S_r$ then $h^2$ is also in $S_r$ since $h^2$ satisfies \eqref{domgenerator} and \eqref{scond}. Furthermore for $h \in S_r$, $h^2 \in D(L_1)$ since $D(L_1)_b$ is an algebra.
 Let $S = \bigcup_{r \in (0,\infty)} S_r$. Since for $h \in S$
\[
L h \in L^{\infty}(\R^d, m)_0,
\] 
$R_1 \big([(1-L)h]^+\big)$, $R_1 \big([(1-L)h]^-\big)$, $R_1 \big([(1-L_1)h^2]^+\big)$, and $R_1 \big([(1-L_1)h^2]^-\big)$ are continuous on $\R^d$ by Proposition \ref{p;strongfl}(i). Furthermore for all $y \in \Q^d$, $\varepsilon \in \Q \cap (0,1)$ we can find $h \in S$ such that $h \ge 1$ on $\overline{B}_{\frac{\varepsilon}{4}}(y)$, $h \equiv 0$ on $\R^d \setminus B_{\frac{\varepsilon}{2}}(y)$. Therefore, we can find a countable subset $\tilde{S} \subset S$ satisfying $(\textbf{H2})^{\prime}$(i) and (ii). Therefore, by Propositions \ref{p;strongfl}(ii) and \ref{l;2.10},  {\bf ($\gamma$)} holds.    
\end{proof}
Consider assumption
\begin{itemize}
\item[{\bf (c)}] $\tilde{c}^{-1} \le \phi \le \tilde{c}$ for some $\tilde{c} \ge 1$.\\
\end{itemize}
\begin{thm}\label{t;Sdec}
Let $\phi$ be as in (\ref{eqphi3}) and suppose {\bf (a)}-{\bf (c)}. Let $\beta := \frac{\overline{\gamma}}{\overline{\gamma} + \gamma}$, $\beta_k  : = \frac{\gamma_{k+1}}{\gamma_{k+1} + \gamma_{k}}$, and $\overline{\beta}_k : = \frac{\overline{\gamma}_{k+1}}{\overline{\gamma}_{k+1} + \overline{\gamma}_k }, \; k \in \mathbb{Z}$. Then the process $\bM$ satisfies
$$
X_{t}^{j} = x_{j} + W_{t}^{j} + \frac{\alpha}{2} \int_{0}^{t} X_s^j \|X_s\|^{-2} \, ds, \;  \ \  j = 1, \dots , d-1,
$$
\begin{equation}\label{LSFD2}
X_{t}^{d} = x _{d}+ W_{t}^{d} + \frac{\alpha}{2} \int_{0}^{t} X_s^d \|X_s\|^{-2} \, ds + \int_{\R} \ell_t^a(X^{d}) \, \mu(da), \; \;\;t  \ge 0,
\end{equation}
$ \P_x $ -a.s. for any $x \in \R^d$, where $(W^1, \dots, W^d)$ is a standard d-dimensional Brownian motion starting from zero, $\ell_{t}^{a}(X^{d})$ is the symmetric semimartingale local time of $X^{d}$ at $a\in (-\infty, \infty)$ as defined in \cite[VI.(1.25)]{RYor} and
\[
\mu : = \sum_{k \in \mathbb{Z}}  \Big( (2\beta_k -1) \, \delta_{l_k} + (2 \overline{\beta}_k -1) \, \delta_{r_k} \Big) + (2\beta -1) \, \delta_{0}.
\]
\end{thm}
\begin{proof}
 By Proposition \ref{t;eprocess},  {\bf ($\alpha$)}-({\bf $\gamma$}) hold. Assumption {\bf (c)} then implies {\bf $(\delta)$}, thus {\bf $(\alpha)$}-{\bf $(\delta)$} hold. {\bf (IBP)} follows from \cite[Proposition 3.11]{ShTr13a}. 
Thus Theorem \ref{t;Ldec} applies. The identification of the drift part in (\ref{LSFD2}) as sum of semimartingale local times then follows as in \cite[proof of Theorem 3.14]{ShTr13a}.
\end{proof}
\begin{remark}\label{r;3.15}
Similarly to the proof of Proposition \ref{p;3.8}(ii) and Theorem \ref{t;3.9}(ii), we can also obtain Theorem \ref{t;Sdec} for $\alpha \in [1,d)$, $d\ge 2$, but then only for all starting points in $\R^d \setminus \{0\}$.
\end{remark}

\paragraph{Normal reflection}\label{s;3.113noref}

This subsection is another example for the application of elliptic regularity results as in Subsection \ref{ss;ermj} and a further example for the localization procedure of Subsection \ref{s;3.1sfdm}. For details, we refer to \cite[Section 5]{ShTr13a}.\\
Let $G \subset \R^d$, $d \ge 2$ be a relatively compact open set with Lipschitz boundary $\partial G$. Suppose
\begin{itemize}
\item[{\bf ($\eta$)}] $\rho = \xi^2$, $\xi \in H^{1,2}(G,dx)\cap C(\overline{G})$ and $\rho>0$ $dx$-a.e. on $G$
\end{itemize}
Then by \cite[Lemma 1.1(ii)]{Tr1}
\[
\E(f,g): = \frac{1}{2} \int_{G} \nabla f \cdot \nabla g \, dm, \quad f,g \in C^{\infty}(\overline{G})
\]
is closable in $L^2(G,m)$. The closure $(\E,D(\E))$ is a regular, strongly local and conservative Dirichlet form (cf. \cite{Tr1}). We further assume
\begin{itemize}
\item[{\bf ($\theta$)}] there exists an open set $E \subset \overline{G}$ with Cap($\overline{G} \setminus E$) = 0 such that $(\E,D(\E))$ satisfies the absolute continuity condition on $E$.
\end{itemize}
By ($\theta$), we mean that there exists a Hunt process
\[
\bM = (\Omega , \mathcal{F}, (\mathcal{F}_t)_{t\geq0}, (X_t)_{t\geq0} , (\P_x)_{x\in E\cup \{\Delta\}} )
\]
with transition kernel $p_{t}(x,dy)$ (from $E$ to $E$) and transition kernel density $p_{t}(\cdot,\cdot) \in \mathcal{B}( E \times E)$, i.e. $p_{t}(x,dy)=p_{t}(x,y) \, m(dy)$, such that 
\[
P_{t} f(x) := \int f(y) \ p_t(x,y) \, m(dy), \quad t>0, \ x \in E, \ f \in \mathcal{B}_b(E)  
\]
with trivial extension to $\overline{G}$ is an $m$-version of $T^{\overline{G}}_t f$ for any $f \in \mathcal{B}_b(E)$, and $(T^{\overline{G}}_t)_{t > 0}$ denotes the semigroup associated to $(\E,D(\E))$. In particular $\bM$ is a conservative diffusion on $E$ (see \cite[Theorem 4.5.4]{FOT}). We rely on elliptic regularity results from \cite{BG} which are applicable in our situation because of \cite[Lemma 5.1(ii)]{ShTr13a} (cf. \cite[Remark 5.2]{ShTr13a}:
\begin{remark}\label{r;FG}
In \cite{BG} also unbounded Lipschitz domains are considered and according to \cite[Theorem 1.14]{BG} ($\theta$) holds with $E=(G \cup \Gamma_2) \cap \{\rho > 0\}$  where $\Gamma_2$ is an open subset of $\partial G$ that is locally $C^2$-smooth, provided $\frac{\|\nabla \rho\|}{\rho} \in L^p_{loc}(\overline{G} \cap \{\rho > 0\},m)$ for some $p \ge 2$ with $p > \frac{d}{2}$ and $\emph{Cap}(\overline{G}\setminus E)=0$.
\end{remark}

Since $E$ is open in $\overline{G}$, we can consider the part Dirichlet form $(\E^{E},D(\E^{E}))$ of $(\E , D(\E))$ on $E$.
\begin{lemma}\label{l;part3}
Let $f \in \mathcal{B}_{b}(E)$. Then $P_t f$ is an $m$-version of $T_{t}^{E}f$.
\end{lemma}
By Lemma $\ref{l;part3}$ the Hunt process $\bM$ is associated with $(\E^{E},D(\E^{E}))$ and satisfies the absolute continuity condition. In addition to {\bf ($\eta$)} and {\bf ($\theta$)}, we assume
\begin{itemize}
\item[{\bf ($\iota$)}] there exists  an increasing sequence of relatively compact open sets  $\{B_k\}_{ k \in \N} \subset E$ such that $\partial B_k, \; k \in \N$ is Lipschitz, $\bigcup_{k \ge 1} B_k = E$ and $ \rho \in (d_k^{-1}, d_k)$ on $B_k$ where $d_k \rightarrow \infty$ as $k \rightarrow \infty$.
\end{itemize}
Considering the part Dirichlet forms on each $B_k$, we obtain the following integration by parts formula (cf. \cite[proof of Theorem 5.4]{Tr1}):
\begin{lemma}\label{ibp3}
For $f \in \{f^1,\dots, f^d\} $ and $g \in C_0^{\infty}(B_k)$, it holds
\begin{equation*}
- \E^{B_k}(f,g)=  \frac{1}{2}   \  \int_{B_k} \left( \nabla f \ \cdot \frac{\nabla \rho  }{\rho}   \right) g \, dm
+  \frac{1}{2} \ \int_{B_k \cap \partial G} \ \nabla f  \cdot \eta   \ g \ \rho \ d\sigma,
\end{equation*}
where $\eta$ is a unit inward normal vector on $B_k \cap \partial G$ and $\sigma$ is the surface measure on $\partial G$.
\end{lemma}
According to \cite{Tr1} the closure of 
\[
\E^{\overline{B}_k}(f,g) = \frac{1}{2} \int_{B_k} \nabla f \cdot \nabla g \, dm, \quad f,g \in C^{\infty}(\overline{B}_k),
\]
in $L^2(\overline{B}_k,m) \equiv L^2(B_k,m)$, $k \ge 1$, denoted by $(\E^{\overline{B}_k},D(\E^{\overline{B}_k}))$, is a regular conservative Dirichlet form on $\overline{B}_k$ and moreover, it satisfies Nash type inequalities as in Lemma \ref{T;NI1} (cf. \cite[Lemma 5.4]{ShTr13a}). Therefore, we obatain estimates for 
$r_1^{B_k}(\cdot,\cdot)$ as in Corollary \ref{C;RDE1}. Using these estimates, Proposition \ref{p;smooth} and Lemma \ref{l;miz}, we obtain the following:
\begin{lemma}\label{t;sm3}
(i) $1_{B_k \cap \partial G} \cdot \rho  \sigma \in S_{00}^{B_k}$.\\
(ii) Let $f \in L^{\frac{d}{2} + \varepsilon} ( B_k , dx)$ for some $\varepsilon > 0$. Then
\[
1_{B_k} \cdot |f|  m \in S^{B_k}_{00}.
\]
In particular  $1_{B_k} \cdot \| \nabla \rho \| dx \in S_{00}^{B_k}$ for $d=2,3$ and for $d \ge 4$, if $\|\nabla \rho\| \in L^{\frac{d}{2} +\varepsilon}(B_k,dx)$ for some $\varepsilon > 0$. 
\end{lemma}

In view of Lemma \ref{t;sm3} (ii), we assume from now on\\
\begin{itemize}
\item[{\bf ($\kappa$)}] If $d \ge 4$ and  $k \ge 1$, then $\|\nabla \rho \| \in L^{\frac{d}{2} + \varepsilon_k} (B_k,dx)$ for some $\varepsilon_k > 0$.\\
\end{itemize}
Applying Proposition \ref{p;LSFD} to the part Dirichlet form $(\E^{B_k},D(\E^{B_k}))$, we get:
\begin{prop}\label{t;lsfd3}
The process $\bM$ satisfies
\begin{equation}\label{lsfd3}
X_{t} = x + W_{t} + \int_{0}^{t}\frac{\nabla \rho}{2\, \rho}(X_s) \, ds + \int^t_0  \eta(X_s) \, d\ell_s^k \quad t < D_{B_k^c}
\end{equation}
$\P_x$-a.s. for any $x \in B_k$ where $W$ is a standard d-dimensional Brownian motion starting from zero and $\ell^k$ is the positive continuous additive functional of $X^{B_k}$  in the strict sense associated via the Revuz correspondence (cf. \cite[Theorem 5.1.3]{FOT}) with the weighted surface measure $\frac{1}{2} \rho  \sigma$ on $B_k \cap \partial G$.
\end{prop}
The proofs of the following two lemmas can be found  in \cite[Section 5]{ShTr13a}. 
\begin{lemma}\label{l;limit2}
$\P_x \big(\lim_{k \rightarrow \infty} D_{B_{k}^c} =\infty) =\P_x \big(\lim_{k \rightarrow \infty} \sigma_{B_{k}^c} =\infty \big)=1$ for all $x \in E$.
\end{lemma}
\begin{lemma}\label{T;LOCAL}
$\ell_t^{k} =  \ell_t^{k+1}, \; \forall t < \sigma_{B_{k}^{c}}$ $\P_{x}$-a.s. for all $x \in B_{k}$ where $\ell_t^{k}$ is the positive continuous additive functional of $X^{B_k}$  in the strict sense associated to $1_{B_k} \cdot \frac{ \rho   \sigma}{2}  \in S_{00}^{B_{k}}$. In particular $\ell_t : = \lim_{k \rightarrow \infty} \ell_t^{k}$, $t \ge 0$, is well defined in $A_{c,1}^{+,E}$, and related to $\frac{\rho \sigma}{2}$ via the Revuz correspondence.
\end{lemma}
Letting $k \rightarrow \infty$ in (\ref{lsfd3}), we finally obtain (cf. the localization procedure of Section \ref{s;3.1sfdm}): 
\begin{thm}\label{t;lsfd4}
The process $\bM$ satisfies
\begin{equation*}
X_{t} = x + W_{t} + \int_{0}^{t}\frac{\nabla \rho}{2\, \rho}(X_s) \, ds + \int^t_0  \eta(X_s) \, d\ell_s \;,\;\;t  \ge 0
\end{equation*}
$\P_x$-a.s. for all $x \in E$ where $W$ is a standard d-dimensional Brownian motion starting from zero and $\ell$ is the positive continuous additive functional of $X$  in the strict sense associated via the Revuz correspondence (cf. \cite[Theorem 5.1.3]{FOT}) with the weighted surface measure  $\frac{1}{2} \rho \sigma$ on $E \cap \partial G$.
\end{thm}

\subsubsection{Degenerate elliptic forms and 2-admissible weights}\label{s;2admine}
In this subsection, we consider a  2-admissible weight $\rho$ (see \cite[Section 1.1]{HKM}) which is strictly positive, i.e. $\rho>0$ dx-a.e. and we let $m:=\rho\,dx$. We assume:
\begin{itemize}
\item[{\bf{(HP1)}}]  
$A=(a_{ij})_{1 \le i,j \le d}$ is a (possibly) degenerate symmetric $d \times d$ matrix of functions $a_{ij} \in L^1_{loc} (\R^d,dx)$ and there exists a constant $\lambda \ge 1$ such that for $dx$-a.e. $x \in \R^d$
\begin{equation}\label{ch5;eq;uelliptic} 
\lambda^{-1} \ \rho(x) \ \| \xi \|^2 \le \langle A(x) \xi, \xi \rangle \le \lambda \ \rho(x) \ \|\xi\|^2, \quad  \forall \xi \in \R^d.
\end{equation} 
\end{itemize}
By \eqref{ch5;eq;uelliptic} and the properties of 2-admissible weights, the symmetric bilinear form
\[
\E^A(f,g) = \frac{1}{2} \int_{\R^d} \langle A  \nabla f , \nabla g \rangle \, dx, \quad f,g \in C_0^{\infty}(\R^d)
\]
is closable in $L^2(\R^d, m)$ and the closure $(\E^A,D(\E^A))$ is a strongly local, regular, symmetric Dirichlet form. Note that $(\E^A,D(\E^A))$ can be written as 
\[
\E^A(f,g) = \frac{1}{2} \int_{\R^d} \Gamma^{\,\rho^{-1}A} (f,g)\,dm, \quad f,g \in D(\E^A),
\]
where $\Gamma^{\rho^{-1}A}$ is a carr\'e du champ (cf. Section \ref{s2;stmg}). We assume from now on
\begin{itemize}
\item[{\bf{(HP2)}}]  
either \ $\sqrt{\rho} \in H^{1,2}_{loc}(\R^d,dx)$\ \  or \ $\rho^{-1} \in L^1_{loc}(\R^d,dx)$.
\end{itemize}
Then the following holds:
\begin{lemma}\label{ch5;l;intrinsic}
For any $x,y \in \R^d$
\begin{equation}\label{ch5;eq;intrinsicm}
\frac{1}{\sqrt{\lambda}} \ \| x-y \| \le \gamma (x,y) \le \sqrt{\lambda} \ \| x-y \|   ,
\end{equation}
where $\lambda \in [1, \infty)$ is as in \eqref{ch5;eq;uelliptic}.
\end{lemma}

\begin{remark}\label{ch5;rebmea}
\begin{itemize}
\item[(i)] Assumption {\bf{(HP2)}} is only used to show that the intrinsic metric of  $\E^A$ with $A=(\rho \delta_{ij})_{1\le i,j\le d}$ is the Euclidean metric, so that the second inequality in (\ref{ch5;eq;intrinsicm}) can be obtained (see \cite[proof of Lemma 2.2]{ShTr15}). It can hence be replaced by any other assumption that implies the fact mentioned above.
\item[(ii)] By \eqref{ch5;eq;intrinsicm}, the intrinsic balls $\tilde{B}_r(x)$, $x \in \R^d$, $r>0$,  are all bounded and open in the Euclidean topology, i.e. they have compact closure.
\end{itemize}
\end{remark}
Since $\rho$ is 2-admissible, it satisfies by definition the doubling property and the scaled weak Poincar\'{e} inequality with respect to the Euclidean metric. With the help of \eqref{ch5;eq;intrinsicm} one can then show that these properties are also satisfied with respect to the intrinsic metric $\gamma(\cdot,\cdot)$ (cf. \cite[Lemmas 2.4 and 2.8]{ShTr15}).  Therefore, the properties of Definition \ref{p;dcpo}(i)-(iv) are satisfied on $\R^d$. In particular, we obtain the existence of a transition semigroup $(P_t)_{t > 0}$ with H\"older continuous heat kernel (see Subsection \ref{s2;stmg}) and Theorems \ref{t;2.7she} and \ref{t;cstconsv} apply. By the latter and \eqref{ch5;eq;intrinsicm}, one can easily see that $\E^A$ is conservative and as in the Subsection \ref{sdbm1}, Proposition \ref{p;strongfl} applies, so that $(P_t)_{t \ge 0}$ is in particular strong Feller. \\
$2$-admissible weights arise typically as in the following example:
\begin{example}\label{ch5;exam2ad} (cf. \cite[Chapter 15]{HKM})
\begin{itemize}
\item[(i)] If $\rho \in A_2$, then $\rho$ is a $2$-admissible weight.
\item[(ii)] If $\rho(x) = | \emph{det} F^{\prime} (x) |^{1-2/d}$ where $F$ is a quasi-conformal mapping in $\R^d$, then $\rho$ is a $2$-admissible weight (for the definition see \cite[Section 3]{FKS}). 
\end{itemize}
\end{example}
\begin{remark}
Since $\E^A$ is given by the  carr\'e du champ  $\Gamma^{\rho^{-1}A}$ and by the Example \ref{ch5;exam2ad} we see that compared to Subsection \ref{sdbm1} the improvement is that we can consider a uniformly strictly globally elliptic diffusion matrix $\rho^{-1}A$ and more general weights $\rho$.
\end{remark}
According to Remark \ref{remhkenonex}, we will now choose an explicit 2-admissible weight. By Example \ref{ch5;exam2ad} a concrete 2-admissible weight that satisfies {\bf{(HP2)}} is given by
 \begin{equation}\label{ch5;eq;2admie}
 \rho(x) = \| x\|^{\alpha}, \quad  \alpha \in (-d, \infty), \quad d \ge 2.
 \end{equation}
Indeed, if $\alpha \in (-d,d)$, then $\rho \in A_2$ and if $\alpha \in (-d+2, \infty)$, then $\rho = | \text{det} F^{\prime} |^{1-2/d}$ for some quasi-conformal mapping $F$ (cf. \cite[Section 3]{FKS}). \\
Up to this end we fix $\rho$ as in \eqref{ch5;eq;2admie}. Then, similarly to Lemma \ref{l;Feller}(i), $(P_t)_{t > 0}$ is seen to be a Feller semigroup, in particular also in the case $\alpha\ge d$. Thus $\bf{(H1)}$ and  $\bf{(H2)}$ are satisfied.
\begin{remark}\label{ch5;re2con}
Let $\phi: \R^d \to \R$ be a measurable function such that  $c^{-1} \le \phi(x) \le c$ $dx$-a.e. for some constant $c \ge 1$. Then by  verifying the properties (I)-(IV) of \cite{ShTr15}, we see that $\phi \rho$ is a 2-admissible weight if $\rho$ is a 2-admissible weight. Moreover choosing $\tilde{A} = (\tilde{a}_{ij})$ satisfying {\bf{(HP1)}} for $\rho \equiv 1$ we see that $A: = \phi \rho \tilde{A}$ satisfies \eqref{ch5;eq;uelliptic} with respect to the 2-admissible weight $\phi \rho$. In particular, the framework of this subsection also includes Dirichlet forms given as the closure of 
\[
\frac{1}{2} \int_{\R^d} \langle \tilde{A} \nabla f, \nabla g  \rangle\,  \phi \rho \, dx, \quad f,g \in C_0^{\infty}(\R^d)
\]
on $L^2(\R^d, \phi \rho dx)$.
\end{remark}

\paragraph{Concrete Muckenhoupt $A_2$-weights with polynomial growth}

We first consider the case where 
 \begin{equation}\label{ch5;emucpol}
 \rho(x) = \| x\|^{\alpha}, \quad  \alpha \in (-d, 2), \quad  d \ge 3.
 \end{equation}
Then the explicit heat kernel estimate that we obtain by Theorem \ref{t;2.7she} is by \eqref{ch5;eq;intrinsicm} comparable to the one that we obtain with $\gamma$ being the Euclidean metric. Thus, we obtain the same resolvent kernel estimate 
as in Lemma \ref{l;Feller}(ii). Consider the following assumption
\begin{itemize}
\item[{\bf{(HP3)}}] For each $i,j =1 , \dots ,d$:
\begin{itemize}
\item[(i)] if  $\alpha \in (-d,-d+2]$, $ \frac{\partial_j a_{ij}}{\rho}  \in   L^1_{loc}(\R^d,m) \cap L^q_{loc}(\R^d,dx)$, $0 < 2- \frac{d}{q} <1$,
\item[(ii)] if $\alpha \in (-d+2,0)$, $ \partial_j a_{ij}  \in L^p_{loc}(\R^d,dx)$ with $0 < 2- \alpha - \frac{d}{p} <1$ and $\frac{\partial_j a_{ij}}{\rho}  \in L^q_{loc}(\R^d,dx)$ with $0 < 2- \frac{d}{q} <1$,
\item[(iii)] if $\alpha \in [0,2)$, $\partial_j a_{ij}   \in L^p_{loc}(\R^d,dx)$ with $0 < 2- \alpha - \frac{d}{p} <1$.
\end{itemize}
\end{itemize}
As in Lemma \ref{l;Feller}(ii), (iii), we then obtain (cf. \cite{ShTr15}):
\begin{lemma}\label{ch5;l;smoothn}
Let  $\rho$ be as in \eqref{ch5;emucpol} and $G\subset \R^d$ any relatively compact open set. Assume {\bf{(HP1)}} and {\bf{(HP3)}}. Then for each $i,j =1 , \dots ,d$
\[
1_{G} \cdot \frac{a_{ii}}{\rho}  m \in S_{00}, \quad 1_{G} \cdot \frac{|\partial_j a_{ij} |}{\rho}  m \in S_{00} .
\] 
\end{lemma}
The following integration by parts formula is easily derived for any $g \in C_0^{\infty}(\R^d)$:
\begin{equation}\label{ch5;eq;ibp}
- \E^A(f^i,g)=    \frac{1}{2} \int_{\R^d} \left( \sum_{j=1}^{d}  \frac{\partial_j a_{ij}}{\rho}  \right) g \, dm, \quad 1 \le i \le d.
\end{equation}
Now using Lemma \ref{ch5;l;smoothn}, \eqref{ch5;eq;ibp}, Proposition \ref{p;LSFD} and the conservativeness, we get:
\begin{thm}\label{ch5;t;stfudeco}
 Assume {\bf{(HP1)}},  \eqref{ch5;emucpol} (which in particular implies {\bf{(HP2)}}), and {\bf{(HP3)}}. Then it holds $\P_x$-a.s. for any $x \in \R^d$, $i=1,\dots,d$
\begin{equation}\label{ch5;sfd2}
X_t^i = x^i + \sum_{j=1}^d \int_0^t \frac{\sigma_{ij}}{\sqrt{\rho}} (X_s) \ dW_s^j +   \frac{1}{2} \int^{t}_{0}   \left(\sum_{j=1}^d  \frac{ \partial_j a_{ij}}{\rho}\right) (X_s) \, ds, \quad t \ge 0,
\end{equation}\label{ch5;eq;sfdec}
where $(\sigma_{ij})_{1 \le i,j \le d} =  \sqrt{A} $ is the positive square root of the matrix $A$, $W = (W^1,\dots,W^d)$ is a standard d-dimensional Brownian motion on $\R^d$.
\end{thm}

\paragraph{Concrete weights with polynomial growth induced by quasi-conformal mappings}\label{s;3122c}
Here, we consider
 \begin{equation}\label{ch5;eqmupo2}
 \rho(x) = \| x\|^{\alpha}, \quad  \alpha \in [2, \infty), \quad d \ge 2.
 \end{equation}
In this case the resolvent kernel estimate of Lemma \ref{l;Feller}(ii) may be not good enough and moreover, we are able to allow better integrability conditions (see {\bf{(HP3)}}$^{\prime}$ below) by using the localization procedure. 
By \cite[Example 3.3.2, Lemma 2.2.7 (ii)]{FOT} and \eqref{ch5;eq;uelliptic}, Cap$(\{0\}) = 0$. Let
\begin{equation}\label{ch5;eq;defbk}
B_k : = \{x \in \R^d  \ | \ (k+1)^{-1} <  \|x \| < k+1 \}, \quad k \ge 1.
\end{equation}
Then condition {\bf (L)} is immediately verified with $\kappa_{B_k}\equiv 1$ and  $d_k=(k+1)^{\alpha}$ for all  $k\ge 1$
Thus for the part Dirichlet forms of $(\E^{A,B_k},D(\E^{A,B_k}))$ of $(\E^A,D(\E^A))$ on $B_k$, we obtain resolvent kernel estimates according to Corollary \ref{C;RDE1}. Due to these improved estimates, we may assume that
\begin{itemize}
\item[{\bf{(HP3)}}$^{\prime}$]   $\partial_j  a_{ij}  \in L_{loc}^{\frac{d}{2} + \varepsilon}(\R^d,dx)$ for some $\varepsilon>0$ and  each $i,j =1 , \dots ,d$, 
\end{itemize}
in order to obtain:
\begin{lemma}\label{ch5;l;smooloc}
Assume {\bf{(HP1)}} and {\bf{(HP3)}}$^{\prime}$. Let  $\rho$ be as in \eqref{ch5;eqmupo2}. Then for each $i,j =1 , \dots ,d$
\[
1_{B_k} \cdot \frac{a_{ii}}{\rho}  m  \in S_{00}^{B_k}, \quad 1_{B_k} \cdot   \frac{|\partial_j a_{ij} |}{\rho}  m \in S_{00}^{B_k}.
\] 
\end{lemma}
From \eqref{ch5;eq;ibp}, we obtain for the coordinate projections $f^i\in D(\E^{A,B_k})_{b,loc}$, $i=1,\dots,d$ and $g \in C_0^{\infty}(B_k)$ 
\begin{equation}\label{ch5;pibp}
- \E^{A,B_k}(f^i,g)=    \frac{1}{2}  \int_{B_k}  \left(  \sum_{j=1}^{d} \frac{\partial_j a_{ij}}{\rho} \right)  g \, dm.
\end{equation}
Then by Lemma \ref{ch5;l;smooloc}, \eqref{ch5;pibp} and Proposition \ref{p;LSFD} applied to the part process, we have:
\begin{prop}\label{t;c5lsfd3}
Assume {\bf{(HP1)}},  \eqref{ch5;eqmupo2}, and  {\bf{(HP3)}}$^{\prime}$. Then the process $\bM$ satisfies \eqref{ch5;sfd2} up 
to $t < D_{B_k^c}$, $\P_x$-a.s. for any $x \in B_k$.
\end{prop}
Since Lemma \ref{l;limit} holds, we finally obtain:
\begin{thm}\label{ch5;t;solex0}
 Assume {\bf{(HP1)}}, \eqref{ch5;eqmupo2}, and  {\bf{(HP3)}}$^{\prime}$. Then the process $\bM$ satisfies \eqref{ch5;sfd2} for all $x \in \R^d \setminus \{0\}$.
\end{thm}

\begin{remark}\label{ch5;reexshtr}
The results of this subsection include the particular case where $\phi \equiv 1$ in Remark \ref{ch5;re2con} with
\begin{equation}\label{ch5;aijtic}
a_{ij}(x) = \tilde{a}_{ij}(x) \|x\|^{\alpha}, \quad \alpha \in (-d,\infty), \quad 1 \le i,j \le d.
\end{equation}
This leads hence to an extension of the results of \cite[Section 3.1 and 3.2]{ShTr13a} with $\phi \equiv 1$ there to the $(a_{ij})$-case. In particular, even if $\tilde{a}_{ij} = \delta_{ij}$ (where $\delta_{ij}$ we obtain partial improvements of results of \cite[Section 3]{ShTr13a}. For instance by our results it is easy to see that in case $\phi \equiv 1$ \cite[Proposition 3.8 (ii)]{ShTr13a} also holds for $\alpha \in [d, \infty)$, $d \ge 2$. Moreover, in view of Remark \ref{ch5;re2con} and the results of this section, it is also possible to extend the results of \cite[Section 3.1 and 3.2]{ShTr13a} to the $(a_{ij})$-case with discontinuous $\phi$, $(a_{ij})$ as in \eqref{ch5;aijtic} satisfying {\bf{(HP3)}}, resp. {\bf{(HP3)}}$^{\prime}$.
\end{remark}

\subsection{Stochastic calculus for the identification of the SDEs}\label{s;scftioft}
\subsubsection{Non-symmetric distorted Brownian motion}\label{s;nsdbm}
This subsection is a continuation of Subsection \ref{2.4.2.2}, where a Hunt process $\bM$ as stated in Theorem \ref{existhunt} was constructed under the assumptions {\bf (A1)-(A3)} and (\ref{aidass}) of Subsection \ref{ss;ermj}. 
We assume throughout this subsection that  {\bf (A1)-(A3)} and (\ref{aidass}) hold. We further consider 
\begin{itemize}
\item[{\bf (A4)}] $(\E,D(\E))$ is conservative.
\end{itemize}

Following \cite[Proposition 3.8]{AKR}, we obtain:
\begin{prop}\label{prop1.5} 
If {\bf (A4)} holds additionally (to {\bf (A1)-(A3)} and (\ref{aidass})), then:
\begin{itemize}
\item[(i)] $\alpha R_{\alpha} 1(x) = 1$ for all $x \in E$, $\alpha > 0$.
\item[(ii)] $(P_t)_{t>0}$ is strong Feller on $E$, i.e. $P_t(\mathcal{B}_b(\R^d)) \subset C_b(E)$ for all $t>0$.
\item[(iii)] $P_t 1(x) = 1$ for all $x \in E$, $t>0$.
\end{itemize}
\end{prop} 
Following \cite[Lemma 5.1]{AKR}, we have:
\begin{lemma}\label{lem3.1}
\begin{itemize}
\item[(i)] Let $f \in \bigcup_{s \in [p, \infty)} L^s(E,m)$, $f \ge 0$, then for all $t > 0$, $x \in E$,
\[
\int_0^t P_s f(x) \ ds < \infty,
\]
hence
\[
\int \int_0^t f(X_s) \ ds \ d\P_x < \infty.
\]
\item[(ii)] Let $u \in C_0^{\infty} (\R^d)$, $\alpha > 0$. Then
\[
R_{\alpha} \big( (\alpha - L) u\big)(x) = u(x) \quad \forall x \in E. 
\]
\item[(iii)] Let $u \in C_0^{\infty}(\R^d)$, $t>0$. Then
\[
P_t u(x) - u(x) = \int_0^t P_s(Lu)(x) \ ds \quad \forall x \in E.
\]
\end{itemize}
\end{lemma}

The following is an immediate consequence of (\ref{operatorrepresentation}):
\begin{lemma}\label{lem1.7}
For $u \in C_0^{\infty} (\R^d)$
\[
L u^2 - 2 u \ Lu = \|\nabla u\|^2.
\]
\end{lemma}
Using in particular Lemma \ref{lem3.1} and Lemma \ref{lem1.7}, we obtain:
\begin{proposition}\label{thm3.2}
Let $u \in C_0^{\infty}(\R^d)$. Then 
$$
M_t^u : = u(X_t) - u(X_0) - \int_0^t Lu(X_r) \, dr, \ t\ge 0,
$$
and
\begin{equation*}
K_t^u : = \left( u(X_t) - u(X_0) - \int_0^t Lu(X_r) \, dr \right)^2 - \int_0^t \|\nabla u\|^2 (X_r) \ dr, \quad t \ge 0. 
\end{equation*}
are continuous $(\mathcal{F}_t)_{t \ge 0}$-martingales under $\P_x$, $\forall x \in E$.
\end{proposition}
\begin{proof}
First one shows that $M_t^u : = u(X_t) - u(X_0) - \int_0^t Lu(X_r) \, dr$, $t\ge 0$, is  a continuous $(\mathcal{F}_t)_{t \ge 0}$-martingale under $\P_x$, $\forall x \in E$. Consequently, there exist stopping times $R_n\nearrow \infty$, such that 
$(M_{t\wedge R_n}^u)_{t\ge 0}$ is a bounded continuous martingale for any $n$ and exactly as in \cite[Appendix]{RoShTr}, we show that $(K_{t\wedge R_n}^u)_{t\ge 0}$
is  a continuous $(\mathcal{F}_t)_{t \ge 0}$-martingale under $\P_x$, $\forall x \in E$. The assertion then follows by letting $n\to \infty$. 
\end{proof}
Proposition \ref{thm3.2} serves to identify the quadratic variation of $M^u$, $u \in C_0^{\infty}(\R^d)$, and subsequently the corresponding SDE. The coordinate functions are smooth functions and hence coincide locally with $C_0^{\infty}(\R^d)$-functions. We will use Proposition \ref{thm3.2} locally up to a sequence of stopping times. For this, we need:
\begin{lemma}\label{pointwisenest}
Let $(B_k)_{k \ge 1}$ be an increasing sequence of relatively compact open sets in $E$ with $\cup_{k \ge 1} B_k= E$.
Then
for all $x \in E$
\[
\P_x \Big(\lim_{k \rightarrow \infty} \sigma_{E \setminus B_{k}} \ge \zeta \Big)=1.
\] 
\end{lemma}
By choosing $(B_k)_{k \ge 1}$ as in Lemma \ref{pointwisenest} to satisfy additionally $\overline{B}_k\subset B_{k+1}$, $k\ge 1$, we can identify (\ref{weaksolution}) with the help of Proposition \ref{thm3.2} for $t<\sigma_{E\setminus B_k}$, $\P_x$-a.s. for any $x\in B_k$. Since this holds for any $k\ge 1$, we can let $k\to \infty$ and obtain (cf. \cite[Theorem 3.6]{RoShTr}):

\begin{thm}\label{t3.6}
After enlarging the stochastic basis $(\Omega, \F, (\F_t)_{t\ge 0},\P_x )$ appropriately for every $x\in E$, the process $\bM$ satisfies 
\begin{eqnarray}\label{weaksolution}
X_t = x + W_t +  \int_0^t  \left( \frac{\nabla \rho}{2 \rho} + B \right) (X_s) \ ds, \quad t < \zeta
\end{eqnarray}
$\P_x$-a.s. for all $x \in E$ where $W$ is a standard $d$-dimensional $(\F_t)$-Brownian motion on $E$. If additionally {\bf (A4)} holds, 
then we do not need to enlarge the stochastic basis and $\zeta$ can be replaced by $\infty$ (cf. Remark \ref{r;connoex}).
\end{thm}

\section{Applications to strong uniqueness of the SDEs}\label{s;strongun}
\subsection{(Non)-symmetric distorted Brownian motion}\label{s;4.1}
This subsection is a continuation of Subsection \ref{s;nsdbm}. We first recall that by \cite[Theorem 2.1]{KR} under the conditions {\bf (A1)}, {\bf (A2)} and (\ref{aidass}) of Subsection \ref{ss;ermj} ({\bf (A3)} is not needed), for every stochastic basis and given Brownian motion $(W_t)_{t\geq0}$ there 
exists a strong solution to \eqref{weaksolution} which is pathwise unique
among all solutions satisfying
\begin{equation}\label{eq4.1}
 \int_0^t \left\| \left( \frac{\nabla \rho}{2 \rho} + B \right)(X_s)\right\|^2 \mathrm{d} s <\infty \quad \mathbb{P}_x\textnormal{-a.s. on } \{t<\zeta\}\;.
\end{equation}
In addition, one has pathwise uniqueness and weak uniqueness in this class. In the situation of Theorem \ref{t3.6} it follows, however immediately from Lemma \ref{pointwisenest} that \eqref{eq4.1} holds for the solution there.  Indeed, by Lemma \ref{pointwisenest}, (\ref{eq4.1}) holds with $\sigma_{E\setminus B_k}$ 
for all $k\in \N$. But the latter together with {\bf (A1)} clearly implies that (\ref{eq4.1}) holds $\P_x$-a.s. for all $x\in S$ for some $S\in {\cal B}(E)$ with $m(E\setminus S)=0$  (by Lemma \ref{lem2.8} the set $S$ can be chosen such that even $\text{Cap}_{\E}(E\setminus S)=0$). So, \cite[Theorem 2.1]{KR}, in particular, implies that the law of 
$\tilde \P_x$ of the strong solution from that theorem coincides with $\P_x$ for all $x\in S$. But then $\tilde \P_x= \P_x$ 
for all $x\in E$, because of the strong Feller property of our Markov process given by $(\P_x)_{x\in E}$ and of the 
one from \cite[Theorem 2.1]{KR}, i.e. $\tilde \P_x$, $x\in E$, since $S$ is dense in $E$. In particular, (\ref{eq4.1}) 
holds for all $x\in E$. Hence we obtain the following (cf. \cite[Theorem 4.1]{RoShTr}):
\begin{thm}\label{t4.1}
Assume {\bf (A1)-(A3)}  and (\ref{aidass}). 
For every $x \in E$ the solution in Theorem \ref{t3.6} is strong, pathwise and weak unique. In particular, it is adapted to the filtration $(\mathcal{F}_t^W)_{t\geq0}$ generated by the Brownian motion $(W_t)_{t\geq0}$
 in \eqref{weaksolution}. 
\end{thm}
\begin{remark}\label{r4.2}
(i) By Theorem \ref{t3.6} and \ref{t4.1} we have thus shown that (the closure of) \eqref{df} 
	is the Dirichlet form associated to the Markov processes given by the laws of the (strong) solutions to \eqref{weaksolution}.
  Hence we can use the theory of Dirichlet forms to show further properties of the solutions. \\ 
	(ii) In \cite{KR} also a new non-explosion criterion was proved (hence one obtains {\bf (A4)}), assuming that $\frac{\nabla \rho}{2 \rho}+B$ is the (weak) gradient of a function 
	$\psi$ which is a kind of Lyapunov function for \eqref{weaksolution}.
  The theory of Dirichlet forms provides a number of analytic non-explosion, i.e. conservativeness criteria 
	(hence implying {\bf (A4)}) which are completely different from the usual ones for SDEs and which are checkable in many cases.
  As stressed in (i) such criteria can now be applied to \eqref{weaksolution}. Even the simple case, where $m(\R^d)<\infty$ and $\|B\| \in L^1(\R^d,m)$ which entails {\bf (A4)}, appears to be a new non-explosion condition for \eqref{weaksolution}. For more sophisticated sufficient non-explosion criteria, we refer to \cite{GTr2016} in general and to \cite[Lemma 5.4]{RoShTr} in a concrete example.
\end{remark}

\subsection{Diffusions with 2-admissible weights}\label{s;4.2}
This subsection is a continuation of Subsection \ref{s;2admine}. We consider 
\begin{itemize}
\item[{\bf{(HP4)}}] For each $1 \le i,j \le d$,
\begin{itemize}
\item[(i)]   $ \frac{\sigma_{ij}}{\sqrt{\rho}}$ is continuous on $\R^d$.
\item[(ii)]  $\left \| \nabla \left(  \frac{\sigma_{ij}}{\sqrt{\rho}} \right)  \right \| \in L^{2(d+1)}_{loc} (\R^d,dx)$.
\item[(iii)]  $\sum_{k=1}^{d}  \frac{ \partial_k a_{ik}}{\rho} \in L^{2(d+1)}_{loc} (\R^d,dx)$.
\end{itemize}
\end{itemize}

\begin{thm}\label{ch5;t;ssoleae} (cf. \cite[Theorem 5.1]{ShTr15})
Assume that {\bf{(HP1)}}, \eqref{ch5;emucpol}, {\bf{(HP3)}}, and {\bf{(HP4)}}, resp. {\bf{(HP1)}}, \eqref{ch5;eqmupo2} {\bf{(HP3)}}$^{\prime}$, and {\bf{(HP4)}} hold. Then the (weak) solution in Theorem \ref{ch5;t;stfudeco}, resp. Theorem \ref{ch5;t;solex0} is strong and pathwise unique. In particular, it is adapted to the filtration $(\mathcal{F}_t^W)_{t\geq0}$ generated by the Brownian motion $(W_t)_{t\geq0}$ as in \eqref{ch5;sfd2} and its lifetime is infinite.
\end{thm}
\proof
Assume that {\bf{(HP1)}}, \eqref{ch5;emucpol}, {\bf{(HP3)}}, and {\bf{(HP4)}}, or 
{\bf{(HP1)}}, \eqref{ch5;eqmupo2}, {\bf{(HP3)}}$^{\prime}$, and {\bf{(HP4)}} hold.
By \cite[Theorem 1.1]{Zh}  under {\bf{(HP1)}} and {\bf{(HP4)}} for given Brownian motion $(W_t)_{t\geq0}$, $x \in \R^d$ as in \eqref{ch5;sfd2} there exists a pathwise unique strong solution to \eqref{ch5;sfd2} up to its explosion time. The remaining conditions make sure that the unique strong solution is associated to $(\E^A,D(\E^A))$ and has thus infinite lifetime. Therefore the (weak) solution in Theorem \ref{ch5;t;stfudeco},  resp. Theorem \ref{ch5;t;solex0}, resp. is strong and pathwise unique.
\qed
\begin{remark}\label{ch5;renonexpn} (cf. \cite[Remark 5.2]{ShTr15})
Two non-explosion conditions for strong solutions up to lifetime for a certain class of stochastic differential equations are presented in \cite[Theorem 1.1]{Zh}. For the precise conditions, we refer to \cite{Zh}. By Theorem \ref{ch5;t;ssoleae} and its proof, we know that the solution of \eqref{ch5;sfd2} up to its lifetime fits to the frame of \cite[Theorem 1.1]{Zh}. Therefore, the remaining conditions 
\[
\eqref{ch5;emucpol}, {\bf{(HP3)}} \quad \text{or} \quad \eqref{ch5;eqmupo2}, {\bf{(HP3)}}^{\prime},
\]  
provide additional non-explosion conditions in \cite[Theorem 1.1]{Zh} for solutions of the form \eqref{ch5;sfd2} that satisfy {\bf{(HP1)}} and {\bf{(HP4)}}. 
\end{remark}
{\it ACKNOWLEDGEMENT:} The second named author would like to thank Michael R\"ockner for bringing up the idea to him to apply pointwise weak existence results for diffusions associated with Dirichlet forms to obtain new non-explosion criteria for the  pathwise unique and strong solutions of  \cite{KR}, \cite{Zh},  as it is done in Section \ref{s;strongun}.

\vspace{2cm}
Jiyong Shin\\
School of Mathematics\\
Korea Institute for Advanced Study\\
85 Hoegiro Dongdaemun-gu, \\
Seoul 02445, South Korea, \\
E-mail: yonshin2@kias.re.kr\\ \\
Gerald Trutnau\\
Department of Mathematical Sciences and \\
Research Institute of Mathematics of Seoul National University,\\
1, Gwanak-Ro, Gwanak-Gu \\
Seoul 08826, South Korea,  \\
E-mail: trutnau@snu.ac.kr
\end{document}